\documentclass[a4paper,11pt]{article}
\usepackage{pdfsync}
\usepackage{empheq}

 \usepackage[plainpages=false]{hyperref}
 \usepackage{amsfonts,amsmath,amssymb,amsthm,enumerate}
 \usepackage{latexsym,lscape,rawfonts,mathrsfs,mathtools}
 \usepackage{enumitem}

 \usepackage[all]{xy}
 \usepackage{eufrak}
 \usepackage{makeidx}
 \usepackage{graphicx,psfrag}
 \usepackage{pstool}
 \usepackage{tikz-cd}
 \usepackage{array,tabularx,tabu}

 \usepackage{setspace}

 \usepackage[titletoc,title]{appendix}

\usepackage{txfonts}

 \newcommand{\R}{\ensuremath{\mathbb{R}}}

 \newcommand{\ba}{\begin{align*}}
 \newcommand{\ea}{\end{align*}}
 \newcommand{\na}{\nabla}
\newcommand{\la}{\langle}
\newcommand{\ra}{\rangle}

\newcommand{\ep}{\epsilon}

 \makeatletter
 \def\ExtendSymbol#1#2#3#4#5{\ext@arrow 0099{\arrowfill@#1#2#3}{#4}{#5}}
 
 \makeatother

 \makeatletter
 \def\ExtendSymbol#1#2#3#4#5{\ext@arrow 0099{\arrowfill@#1#2#3}{#4}{#5}}
 \newcommand\longright[2][]{\ExtendSymbol{-}{-}{\rightarrow}{#1}{#2}}
 \makeatother

\def\XXint#1#2#3{{\setbox0=\hbox{$#1{#2#3}{\int}$ }
\vcenter{\hbox{$#2#3$ }}\kern-.55\wd0}}

\numberwithin{equation}{section}

\newtheorem{thm}{Theorem}[section]

\newtheorem{prop}[thm]{Proposition}
\newtheorem{lem}[thm]{Lemma}

\newtheorem{rem}[thm]{Remark}

\newtheorem{defn}[thm]{Definition}

\title{Canonical diffeomorphisms of manifolds near spheres}
\author{Bing Wang, Xinrui Zhao}

\date{}

\begin{document}
\maketitle

\begin{abstract}
For a given Riemannian manifold $(M^n, g)$ which is near standard sphere $(S^n, g_{round})$ in the Gromov-Hausdorff topology and  satisfies $Rc \geq n-1$, 
it is known by Cheeger-Colding theory that $M$ is diffeomorphic to $S^n$. 
A diffeomorphism $\varphi: M \to S^n$ was constructed in \cite{2} using Reifenberg method. 
In this note, we show that a desired diffeomorphism can be constructed canonically. 
Let $\{f_i\}_{i=1}^{n+1}$ be the first $(n+1)$-eigenfunctions of $(M, g)$ and $f=(f_1, f_2, \cdots, f_{n+1})$.
Then the map $\tilde{f}=\frac{f}{|f|}: M \to S^n$ provides a diffeomorphism, and $\tilde{f}$ satisfies a uniform bi-H\"older estimate.
We further show that this bi-H\"older estimate is sharp and cannot be improved to a bi-Lipschitz estimate.
Our study could be considered as a continuation of Colding's work \cite{5}\cite{6} and Petersen's work \cite{13}.
\end{abstract}


\section{Introduction}






For a Riemannian manifold $(M^{n}, g)$ satisfying the Ricci lower bound
\begin{align}
    Rc \geq n-1,    \label{eqn:OH31_8}
\end{align}
there exist many rigidity theorems. 
For example,  the Lichnerowicz-Obata theorem (cf.~\cite{11}~\cite{12}) states that such a manifold must satisfy
\begin{align}
  \lambda_1 \geq n,    \label{eqn:OH30_1}
\end{align}
and the equality holds if and only if $(M, g)$ is isometric to the standard sphere $(S^{n}, g_{round})$ with constant sectional curvature $1$.    
Here $\lambda_1$ is the first eigenvalue of the Laplace operator $-\Delta$. On the other hand, it follows from the Bishop-Gromov volume comparison and Myer's theorem (cf.~\cite{24}) that 
\begin{align}
 Vol(M, g) \leq Vol(S^{n}, g_{round})=(n+1) \omega_{n+1}, \label{eqn:OH30_2}
\end{align}
and the equality holds if and only if $(M,g)$ is isometric to the standard sphere.  Here $\omega_{n+1}$ is the volume of the standard ball in the Euclidean space $\R^{n+1}$. 

In both the Lichnerowicz-Obata theorem and the Bishop-Gromov-Myer theorem, the standard sphere $(S^{n}, g_{round})$ plays the role of a model space. 
If $(M, g)$ is isometric to $(S^{n}, g_{round})$,  then the multiplicity of the eigenvalue $\lambda_1$ is $n+1$.   Namely, we have
\begin{align*}
    \lambda_1=\lambda_2=\cdots=\lambda_{n+1}=n. 
\end{align*}
Let $\{f_1, f_2, \cdots, f_{n+1}\}$ be an orthonormal basis of the eigenspace corresponding to the eigenvalue $n$,  in the sense that
\begin{subequations}
  \begin{empheq}[left = \empheqlbrace \,]{align}
    &\frac{1}{n+1} \fint_{S^{n}} |f_i|^2 =1, \quad \forall  \;1 \leq i \leq n+1;  \label{eqn:OI01_1a} \\
    & \fint_{S^{n}} f_i f_j =0, \quad \forall \; 1 \leq i< j \leq n+1.   \label{eqn:OI01_1b}
  \end{empheq}
\label{eqn:OI01_1}  
\end{subequations}
\hspace{-3mm}
Here $\fint_{S^{n}}$ denotes taking average on $(S^n,g_{round})$. Then the map
\begin{align}
    f \coloneqq (f_1, f_2, \cdots, f_{n+1})   \label{eqn:OI01_2}
\end{align}
is an embedding of $S^{n}$ into $\R^{n+1}$.  More precisely, $f$ is an isometry from $(S^n, g_{round})$ to the hyper-surface
\begin{align}
  \tilde{S} \coloneqq \left\{(x^1, x^2, \cdots, x^{n+1}) \left| (x^{1})^2 + (x^2)^2 + \cdots + (x^{n+1})^2=1  \right. \right\}  \subset \R^{n+1}   
  \label{eqn:OH31_9}
\end{align}
with the metric induced from the Euclidean space. 
Define
\begin{align}
  &pr: \R^{n+1} \backslash \{0\} \mapsto \tilde{S},   \quad      (x^1, x^2, \cdots, x^{n+1})  \mapsto \frac{(x^1, x^2, \cdots, x^{n+1}) }{\sqrt{(x^{1})^2 + (x^2)^2 + \cdots + (x^{n+1})^2}};  \label{eqn:OH31_10} \\
  & \tilde{f} \coloneqq pr \circ f.      \label{eqn:OI01_3}      
\end{align}
Thus $\tilde{f}$ provides an isometry between $(M, g)$ and $(\tilde{S}, g_{E}|_{\tilde{S}}) \cong (S^{n}, g_{round})$. 
In short, if $(M, g)$ is isometric to $(S^{n}, g_{round})$, then we can use eigenfunctions to construct a \textit{canonical} isometry $\tilde{f}$ from $(M, g)$ to $(S^{n}, g_{round})$.

Both rigidity theorems above have quantitative versions.  In the celebrated work~\cite{5}\cite{6},  Colding proved that  the Gromov-Hausdorff  closeness to the standard sphere is equivalent  to the  volume closeness to the standard sphere, under the assumption $Rc \geq n-1$.  
The seminal ideas and powerful techniques  introduced in these papers become one of the major inspirations for the structure theory of spaces with Ricci curvature bounded below, which is now called the Cheeger-Colding theory. 
The quantitative version of the Lichnerowicz-Obata theorem was studied in~\cite{13}.  Under the same assumption $Rc \geq n-1$,  Petersen showed that Gromov-Hausdorff  closeness to a sphere is equivalent  to $\lambda_{n+1}$ closeness to $n$.
He observed that the function $\tilde{f}$ could be defined for each $(M, g)$ satisfying $Rc \geq n-1$ and $\lambda_{n+1} \leq n+\epsilon(n)$, following the same route as in (\ref{eqn:OI01_1})-(\ref{eqn:OI01_3}),
and pointed out that this map should be a Gromov-Hausdorff approximation map.

Our study is motivated by both the works of Colding and of Petersen. We shall apply the techniques from the Cheeger-Colding theory (cf.~\cite{1}\cite{2}\cite{25}\cite{21}\cite{32}\cite{3}\cite{4}) to study the refined property of the map $\tilde{f}$. 
Our result reveals that $\tilde{f}$ is not only a Gromov-Hausdorff approximation, but also a diffeomorphism with uniform bi-H\"older estimate. 

\begin{thm}[\textbf{Main theorem}]
For each dimension $n \geq 3$ and small constant $\epsilon>0$, there exists a small constant $\delta=\delta(n, \epsilon)$ with the following properties.

Suppose $(M^{n}, g)$ is a closed Riemannian manifold satisfying 
\begin{align}
\begin{cases}
&Rc \geq n-1, \\
&Vol(M, g) \geq (n+1) \omega_{n+1}-\delta. 
\end{cases}  
\tag{$\dagger$}
\label{eqn:OI01_4}
\end{align}
Suppose $f_1,...,f_{n+1}$ are the first $(n+1)$-normalized eigenfunctions (cf. (\ref{eqn:OI01_1})) corresponding to eigenvalues $0<\lambda_1\leq...\leq\lambda_{n+1}$ 
of the Laplacian operator $-\Delta$. Then the map
\begin{align}
 \tilde{f} \coloneqq  \frac{f}{|f|}  \coloneqq \frac{(f_1, f_2, \cdots, f_{n+1})}{\sqrt{f_1^2+...+f_{n+1}^2}}
\label{eqn:OH30_3}
\end{align}
is well-defined and provides a diffeomorphism between M and $S^n$.
Furthermore, for each pair of points $x,y \in M$ satisfying $d(x,y) \leq 1$,  the following distance   bi-H\"older estimates hold
\begin{align}
   (1-\epsilon)d(x,y)^{1+\epsilon}\leq  d\left(\tilde{f}(x), \tilde{f}(y) \right) \leq (1+\epsilon)d(x,y), 
\label{eqn:OH30_4} 
\end{align} 
where $S^n$ is equipped with the canonical round metric.  The above uniform bi-H\"older estimates cannot be improved to be uniform bi-Lipschitz estimate. 
\label{thm:main}
\end{thm}

Note that condition (\ref{eqn:OI01_4}) can be replaced by other equivalent conditions, like volume closeness to $(n+1)\omega_{n+1}$, or radius closeness to $\pi$, or $\lambda_{n+1}$ closeness to
$n$, as proven by Colding and Petersen (cf. Theorem~\ref{thm:OH30_5}).
Under one of these conditions, the facts that $M$ is diffeomorphic to $S^{n}$, and $M$ is uniformly bi-H\"older equivalent to $S^{n}$,  were first proved by Cheeger-Colding in~\cite{2} using the Reifenberg method. 
Theorem~\ref{thm:main} has the advantage that it provides a \textit{canonical} diffeomorphism. 
Similar phenomena hold for other canonical diffeomorphisms.  For example,  in~\cite{22}\cite{23},  the first named author used the Ricci flow to construct another canonical diffeomorphism from $M$ to $S^{n}$ with uniform bi-H\"older estimate.
It was also shown there that the bi-H\"older estimate cannot be improved to a bi-Lipschitz map.

Now we briefly describe the structure of the paper and outline the proof of Theorem~\ref{thm:main}. In Section~\ref{sec:pre}, we recall preliminary results which are of fundamental importance for this study.  
In Section~\ref{sec:transformation}, we prove transformation theorems (cf. Theorem~\ref{thm:OH30_10} and Theorem~\ref{thm:OI07_1}) in our setting.   Namely, up to orthogonal transformation and rescaling, in a local domain, the eigenfunctions $\{f_k\}_{k=1}^{n}$ provide a natural $(n,\epsilon)$-splitting map.  A blowup argument then shows that the $(n, \epsilon)$-splitting property will survive in any small scale, up to further transformations by lower triangular matrices.
 We adapt and substantially simplify some techniques from~\cite{3} in our special setting. 
In Section~\ref{sec:canonical},  we use the transformation theorems and the gradient estimate to setup the bi-H\"older estimate for $\tilde{f}$.  Then we apply the bi-H\"older estimate to show the injectivity of $\tilde{f}$ and the 
 non-degeneracy of $d\tilde{f}$.  We also show the surjectivity of $\tilde{f}$  by a mod-$2$ degree argument due to Cheeger~\cite{1}. Consequently $\tilde{f}$ is a diffeomorphism. 
Finally, in Section~\ref{sec:example}, we construct examples to show that the  bi-H\"older estimate is sharp.   The examples are constructed by gluing rotationally symmetric sphere metrics and football metrics in an explicit way. 
We also discuss further applications and generalization of our methods and results towards the end of this paper(cf. Remark~\ref{rmk:OI28_1} and Remark~\ref{rmk:OI28_2}). \\

{\bf Acknowledgements}:  

 Bing Wang would like to thank Shaosai Huang,  Yu Li  and Haozhao Li for helpful discussions.   
 Xinrui Zhao is grateful to Tobias Colding for his inspirational suggestions.  Bing Wang is supported by  NSFC 11971452, NSFC12026251, and YSBR-001.

\section{Preliminaries}
\label{sec:pre}

In this section we collect important results in the literature which will be used in the paper.

The following theorem, stating that several nearby-sphere conditions are equivalent, is due to Colding and Petersen. 

\begin{thm}(\cite{6}\cite{5}\cite{13}) 
For any manifold $(M, g)$ with $Rc \geq n-1$,  the following conditions are equivalent:
\begin{itemize}
\item[(1).] Vol(M) is close to Vol($S^n$);
\item[(2).] rad(M) is close to $\pi$;
\item[(3).] $M$ is Gromov-Hausdorff close to $S^n$;
\item[(4).] $\lambda_{n+1}$ is close to $n$. 
\end{itemize}
\label{thm:OH30_5}
\end{thm}

 For simplicity of notations, we denote $\tau(\ep|\,n,\lambda, \cdots)$ as a nonnegative function which depends on $\ep,\; n,\;\lambda...$ and converges to zero as $\ep\to 0$.
 Therefore, under the condition (\ref{eqn:OI01_4}), we have
 \begin{align}
  &0 \leq \pi -rad(M) \leq  \tau(\delta | n);  \label{eqn:OI09_16}\\
  &d_{GH} \{(M, g), (S^n, g_{round})\} < \tau(\delta | n); \label{eqn:OI09_17}\\
  &0 \leq \lambda_{n+1}-n \leq \tau(\delta | n). \label{eqn:OI09_18}
 \end{align}
In this paper, we let $C=C(n)$. Both $C$ and $\tau(\epsilon| n, \lambda, \cdots)$ may change from line to line.
For a functionn $h: M \to \R$, we define
 \begin{align*}
 \|h \|_p \coloneqq \left(\frac{1}{Vol(M)}\int_M |h|^p\right)^{\frac{1}{p}}. 
 \end{align*}

\begin{lem}(cf. Lemma 3.2 of~\cite{13}) For any manifold $(M, g)$ with $Rc \geq n-1$, assume that $\lambda_k\leq n+\ep$. Then for the corresponding eigenfunctions $f_i$ (i=1,2,...,k), letting $\phi=\Sigma_{i=1}^k \alpha_if_i$ with $|\Sigma_{i=1}^k \alpha_i^2-1|\leq\ep$, we have
\begin{align}
 &\phi^2+|\na \phi|^2\leq 1+\tau(\ep|\,n);  \label{eqn:OG13_1} \\
 &||\phi^2+|\na \phi|^2-1||_p\leq \tau(\ep|\,n,p).  \label{eqn:OG13_2}
\end{align}
\label{lma:OH30_11}
\end{lem}

\begin{lem}(cf. Lemma 4.3 and Lemma 5.2 of~\cite{13}) 
For any manifold $(M, g)$ with $Rc \geq n-1$, assume that $\lambda_{n+1}\leq n+\ep$. Then for $f=(f_1,...,f_{n+1})$ and $\tilde{f}=\frac{f}{|f|}$, we have
\begin{align}
&\left|  |f|^2-1 \right|\leq \tau(\ep|\,n);  \label{eqn:OH30_6} \\
&\left|\tilde{f}(x) \cdot \tilde{f}(x_0)-\cos \,d(x,x_0) \right|\leq\tau(\ep|\,n); \label{eqn:OH30_7}\\
&\left|d\tilde{f}-df \right|\leq \tau(\ep|\,n); \label{eqn:OH30_8}\\
&\sup_{\{v \in TM| |v|=1\}}\left|d\tilde{f}(v) \right| \leq 1+\tau(\ep|\,n). \label{eqn:OH30_9}
\end{align}
\label{lma:OH30_12}
\end{lem}

It is clear that all the inequalities in Lemma~\ref{lma:OH30_11} and Lemma~\ref{lma:OH30_12} are equalities for standard sphere. 
Note that (\ref{eqn:OH30_9}) is a gradient estimate,  which bounds the length of $d\tilde{f}(v)$ from above. 
However, the lower bound for $\displaystyle \inf_{\{v \in TM| |v|=1\}}\left|d\tilde{f}(v) \right| $ is not known. 
Actually,  we shall show in the final section(cf. Proposition~\ref{prn:OI19_1}) that a uniform positive lower bound  for $\displaystyle \inf_{\{v \in TM| |v|=1\}}\left|d\tilde{f}(v) \right| $ is impossible.

\begin{defn}
A smooth map $\mathbf{u}=(u_1, u_2, \cdots, u_k):\ B(x,r)\rightarrow \R^k$ is called a $(k, \epsilon)$-splitting if the following conditions are satisfied. 
\begin{itemize}
\item[(1).] $\mathbf{u}(x)=0$; 
\item[(2).] $\sup_{B(x,r)}|\na \mathbf{u}|^2 \leq k+\epsilon$;
\item[(3).] $\fint_{B(x,r)}|\langle\na u_{\alpha}, \na u_{\beta} \rangle-\delta_{\alpha\beta}| \leq \epsilon$ for each $1 \leq \alpha \leq \beta \leq k$;
\item[(4).] $r^2\fint_{B(x,r)} |\na^2 \mathbf{u}|^2 \leq \epsilon^2$. 
\end{itemize}
\label{dfn:OH30_14}
\end{defn}

The $(k, \epsilon)$-splitting map was introduced by Cheeger-Colding~\cite{26}. 
If $\epsilon=0$, then $\mathbf{u}$ exactly provides a splitting. For recent study in this direction, see~\cite{17} for example.  
In the special case that $k=n=\dim M$, we have the following lemma. 

\begin{lem}[\cite{21}]
Suppose $\mathbf{u}=(u_1,...,u_n): B(z,10)\rightarrow \R^n$ is smooth on a manifold $M^n$ with $Rc\geq -(n-1)$  such that
\begin{align*}
 &\sup\limits_{B(z,3)}|\na \mathbf{u}|\leq 2\sqrt{n}, \\
 &\fint_{B(z,10)}\left\{\sum\limits_{1\leq i\leq n}|\na u_i-1|^2+\sum_{1\leq i< j\leq n}|\langle\na u_i, \na u_j\rangle|+\sum\limits_{1\leq i\leq n}|Hess \,\,u_i|^2\right\}\leq \zeta.
\end{align*}
Then we have
\begin{align*}
\sup_{y\in \overline{B(z,1)}} \left||\mathbf{u}(y)-\mathbf{u}(z)|^2-d^2(y,z) \right|\leq C(n)\zeta^{\frac{1}{n+2}}.
\end{align*}
\label{lma:OI10_1}
\end{lem}

Lemma~\ref{lma:OI10_1} originated from~\cite{21}.  It was generalized to spaces with mild singularities in~\cite{4}.
The statement here follows the one in Lemma 4.7 of the ``Futher details"  of~\cite{4}. 
The key point is to apply the segment inequality, as done in~\cite{1}.  Note that the number of maps and the dimension of $M^n$ should be the same.

\begin{thm}[$W^{1,2}$-convergence,~\cite{8}]
Let $(M^n_i,g_i,x_i,\mu_i)\rightarrow (X,d,x_{\infty},\mu)$ with $Ric_{M^n_i}\geq -(n-1)$ and $\mu_i=\frac{Vol}{Vol(B(x_i, 1))}$. 
Let $u_i: B(x_i, R) \rightarrow \mathbb{R}$ be smooth functions satisfying 
\begin{align*}
\fint_{B_R(x_i)}|u_i|^2+|\nabla u_i|^2+|\Delta u_i|^2\leq C.
\end{align*}
If $u_i$ converges in the $L^2$-sense to a $W^{1,2}$-function $u_{\infty}: B(x_{\infty}, R) \rightarrow \mathbb{R}$, then the following properties hold. 

\begin{itemize}
\item[(1).] $u_i \longright{W^{1,2}} u_{\infty}$ over $B(x_\infty, R)$;
\item[(2).] $\Delta u_i \longright{weakly-L^2}  \Delta u_{\infty}$ over $B(x_\infty, R)$;
\item[(3).] If $\displaystyle \sup_{B(x_i, R)}|\na u_i|$ is uniformly bounded by $L$, then $u_i \longright{W^{1,p}} u_\infty$ over $B(x_\infty, R)$ for any $p \in (0, \infty)$.
\end{itemize}
\label{thm:OH31_4}
\end{thm}

\section{Transformation theorems}
\label{sec:transformation}

If $M$ is isometric to the standard sphere, then the first $(n+1)$-normalized eigenfunctions could be chosen as Cartesian coordinate functions
$x_i \;(i=1,2,...,n+1)$.  For any point $x$ in $M$, composing an orthogonal linear transformation $T$, we can rotate $x$ to the north pole.
For simplicity of notation, we denote the north pole by $q$.  Namely, we have
 \begin{align}
 q \coloneqq  (0,0,...,0,1) = T\tilde{f}(x).   \label{eqn:OI10_2}
 \end{align}
Now we define
 \begin{align*}
  \mathbf{F} \coloneqq ((T\tilde{f})_1,..(T\tilde{f})_n).   
 \end{align*}
 Then we have that $\mathbf{F}(x)=0$, $\nabla \mathbf{F}(x)$ is an orthogonal matrix and  $\nabla^2 \mathbf{F}_{\alpha}(x)=n I_{n}$ for each $1 \leq \alpha \leq n$. 
 Therefore, for any $\epsilon>0$, there exists a constant $\delta=\delta(n,\epsilon)$ independent of $x$ and $T$ 
 such that $\mathbf{F}|_{B(x,\delta)}$ is a generalized $\epsilon$-splitting.  
 Similar phenomenon also holds for manifolds nearby sphere.

\begin{thm}[\textbf{Transformation A}]
For any $\epsilon>0$ and $r \in (0, c_n \epsilon)$ for some small dimensional constant $c_n$,  there exists $\delta=\delta(\ep, r, n)$ with the following properties.

Suppose $(M^{n}, g)$ is a closed Riemannian manifold satisfying (\ref{eqn:OI01_4}) and $x\in M$.
Then there exists an orthogonal matrix $T \in SO(n+1, \R)$ such that 
\begin{itemize}
\item[(1).] $\displaystyle T\tilde{f}(x)=q=(0,0,...,0,1)$. 
\item[(2).] $\mathbf{F}|_{B(x,r)}$ is an $(n,\epsilon)$-splitting, where $\displaystyle  \mathbf{F} \coloneqq ((T\tilde{f})_1,..(T\tilde{f})_n)$. 
\end{itemize}
\label{thm:OH30_10}
\end{thm}

 \begin{proof}
 
 Since $SO(n+1, \R)$ acts transitively on the standard sphere, it is clear that we can find a $T \in SO(n+1, \R)$ such that 
 \begin{align}
   T \tilde{f}(x)=q=(0, \cdots, 0, 1).  \label{eqn:OI01_5}
 \end{align}
Note that the choice of $T$ depends on $\tilde{f}(x)$ and may not be unique.  Then we have
\begin{align}
     & \mathbf{F}=(F_1, F_2, \cdots, F_n) \coloneqq ((T\tilde{f})_1, \cdots, (T\tilde{f})_n),   \label{eqn:OI06_3}\\
     & F_{\alpha}=T_{\alpha}^{k} \tilde{f}_k.  \label{eqn:OI06_4}
\end{align}
The coefficients $T_{\alpha}^k$ in (\ref{eqn:OI06_4}) satisfy
\begin{align}
    \sum_{k=1}^{n+1} \left( T_{\alpha}^k \right)^2=1.     \label{eqn:OI06_2}
\end{align}
Recall that $\tilde{f}_k=\frac{f_k}{|f|}$ and $|f|$ is almost $1$ by Lemma~\ref{lma:OH30_12}.   By setting $A_{\alpha}^k \coloneqq \frac{T_{\alpha}^k}{|f|}$, we have
\begin{align}
  & F_{\alpha}=A_{\alpha}^k f_k,  \label{eqn:OI06_5}\\
  & \left| \sum_{k=1}^{n+1} \left( A_{\alpha}^k \right)^2 -1 \right| \leq \tau(\delta | n).   \label{eqn:OI06_6}
\end{align}

 It follows immediately from (\ref{eqn:OI01_5}) and (\ref{eqn:OI06_3}) that  $\mathbf{F}(x)=0$.
 Therefore, condition (1) in Definition~\ref{dfn:OH30_14} holds immediately.  
 It suffices to show condition (2)-(4) in Definition~\ref{dfn:OH30_14}.  We shall do this step by step. 
 
 We let Roman letters $i,j$ run from $1$ to $n+1$, Greek letters $\alpha, \beta$ run from $1$ to $n$.
 We fix $r \in (0,1)$ and estimate the quantities in Definition~\ref{dfn:OH30_14} in the ball $B(x,r) \subset M$. 
 In this proof, $\tau=\tau(\delta | n)$. \\

 \textit{Step 1.  There holds that
 \begin{align}
     \sup_{B(x,r)}|\na \mathbf{F}|^2 \leq n+\tau.     \label{eqn:OI02_3}
 \end{align}
 }

In light of (\ref{eqn:OI06_6}),  we can apply Lemma~\ref{lma:OH30_11} to obtain
\begin{align}
   F_{\alpha}^2 + |\nabla F_{\alpha}|^2 \leq 1 + \tau. \label{eqn:OI03_3}
\end{align}
Summing the above inequality, we obtain
\begin{align*}
  |\nabla \mathbf{F}|^2=\sum_{\alpha=1}^{n} |\nabla F_{\alpha}|^2 \leq \sum_{\alpha=1}^{n} \left( |\nabla F_{\alpha}|^2  + F_{\alpha}^2 \right) < n+\tau, 
\end{align*}
which is exactly (\ref{eqn:OI02_3}). \\

 \textit{Step 2.  For each $	1 \leq \alpha \leq \beta \leq n$,  there holds that
 \begin{align}
    \fint_{B(x,r)}|\langle\na F_{\alpha}, \na F_{\beta} \rangle-\delta_{\alpha \beta}|<C \left\{ r^2 + r^{-n} \tau\right\}.       \label{eqn:OI02_4}
 \end{align}
 }
  
It is clear that (\ref{eqn:OI02_4}) is equivalent to the following two inequalities.
 \begin{align}
   &\fint_{B(x,r)} \left||\nabla F_{\alpha}|^2-1 \right|<C \left\{ r^2 + r^{-n} \tau\right\}, \quad \forall \; 1 \leq \alpha \leq n; \label{eqn:OI02_7}\\
   &\fint_{B(x,r)} \left|\langle\na F_{\alpha}, \na F_{\beta}\rangle \right|<C \left\{ r^2 + r^{-n} \tau\right\}, \quad \forall 1 \leq \alpha <\beta \leq n.   \label{eqn:OI02_8}
 \end{align}
 We first show (\ref{eqn:OI02_7}). 
 Inequality (\ref{eqn:OG13_2}) in Lemma~\ref{lma:OH30_11} implies that
\begin{align}
 \fint_M \left|{F_{\alpha}}^2+|\na F_{\alpha}|^2-1 \right|\leq \tau.   \label{eqn:OI02_9}
\end{align}
By the gradient estimate in (\ref{eqn:OI03_3}),  it is clear that
\begin{align}
   |F_{\alpha}| \leq (1+\tau)r\leq 2r, \quad \textrm{in} \;  B(x,r).     \label{eqn:OI02_10}
\end{align}
By triangle inequality,  we have
\begin{align*}
   \quad \fint_{B(x,r)} \left| |\na F_{\alpha}|^2-1 \right| 
&\leq  \fint_{B(x,r)} F_{\alpha}^2 + \fint_{B(x,r)} \left|{F_{\alpha}}^2+|\na F_{\alpha}|^2-1 \right|\\
&\leq 4r^2 + \frac{Vol(M)}{Vol(B(x,r))}\fint_M \left|{F_{\alpha}}^2+|\na F_{\alpha}|^2-1 \right|,
\end{align*}
which implies (\ref{eqn:OI02_7}) via an application of volume comparison. 

We proceed to show (\ref{eqn:OI02_8}).  In Lemma~\ref{lma:OH30_11}, we can take $\phi=\frac{F_{\alpha}+F_{\beta}}{\sqrt{2}}$ to obtain that
\begin{align*}
  \fint_M \left|\frac{{F_{\alpha}}^2+{F_{\beta}}^2+2F_{\alpha}F_{\beta}}{2}+\frac{|\na F_{\alpha}|^2+|\na F_{\beta}|^2+2\la\na F_{\alpha}, \na F_{\beta}\ra}{2}-1 \right| \leq \tau.
\end{align*}
Plugging (\ref{eqn:OI02_9}) and its version for $F_{\beta}$ into the above inequality,  we can apply triangle inequality to obtain
\begin{align*}
  \fint_M \left|F_{\alpha} F_{\beta}+ \langle \nabla F_{\alpha}, \nabla F_{\beta} \rangle \right|  \leq \tau. 
\end{align*}
Applying (\ref{eqn:OI02_10}) again, we obtain
\begin{align*}
  \fint_{B(x,r)} \left| \langle \nabla F_{\alpha}, \nabla F_{\beta} \rangle \right| &\leq \fint_{B(x,r)} \left|F_{\alpha} F_{\beta}\right| + \fint_{B(x,r)} \left|F_{\alpha} F_{\beta}+ \langle \nabla F_{\alpha}, \nabla F_{\beta} \rangle \right| \\
     &\leq 4r^2 + \frac{Vol(M)}{Vol(B(x,r))} \fint_{M} \left|F_{\alpha} F_{\beta}+ \langle \nabla F_{\alpha}, \nabla F_{\beta} \rangle \right|\\
     &\leq C \left\{ r^2 + r^{-n} \tau\right\}, 
\end{align*}
which is nothing but  (\ref{eqn:OI02_8}).  Therefore, both (\ref{eqn:OI02_7}) and (\ref{eqn:OI02_8}) are proved. We finish the proof of Step 2. \\

 \textit{Step 3. For each $1 \leq \alpha \leq n$,   there holds that
 \begin{align}
     r^2\fint_{B(x,r)} \left|\na^2 F_{\alpha} \right|^2< C \left\{ r^2 + r^{-n} \tau \right\}.     \label{eqn:OI02_5}
 \end{align}
 }
 
Recall that
\begin{align*}
 \tilde{f}_{\alpha}=\frac{f_{\alpha}}{|f|}, \quad   \na \tilde{f}_{\alpha}=\frac{\na f_{\alpha}}{|f|}-\frac{f_{\alpha} f_k\na f_k}{|f|^3}. 
\end{align*}
Direct calculation yields that
\begin{align}
&\quad \; \na^2 \tilde{f}_{\alpha} \notag\\
&=\frac{\na^2 f_{\alpha}}{|f|}-\frac{f_k \left\{\na f_k\otimes \na f_{\alpha}  + \nabla f_{\alpha} \otimes \nabla f_k \right\}}{|f|^3}-\frac{f_{\alpha} \na f_k\otimes \na f_k}{|f|^3}  
  -\frac{f_{\alpha}f_k\na^2 f_k}{|f|^3}+\frac{3f_{\alpha}f_k f_l \na f_k \otimes \na f_l}{|f|^5}.
\label{eqn:OI08_7}  
\end{align} 
Note that $f_{k}^2+|\nabla f_k|^2 \leq 1+\tau$ by Lemma~\ref{lma:OH30_11}.  Thus we have
\begin{align}
    \left| \na^2 \tilde{f}_{\alpha} \right| \leq C \left\{ 1+ \sum_{k=1}^{n+1} |\nabla^2 f_k|^2 \right\}.
     \label{eqn:OI06_1}
\end{align}
Recall that $F_{\alpha}$ is a linear combination of $\tilde{f}_{k}$ with coefficients satisfying (\ref{eqn:OI06_4}) and (\ref{eqn:OI06_2}).  It follows that
\begin{align*}
    \left| \na^2 F_{\alpha} \right|=\left| T_{\alpha}^k \nabla^2 \tilde{f}_k\right|^2 \leq \sum_{k=1}^{n+1} \left( T_{\alpha}^k \right)^2 \left| \nabla^2 \tilde{f}_k \right|^2
    \leq C \left\{ 1+ \sum_{k=1}^{n+1} |\nabla^2 f_k|^2 \right\}, 
\end{align*}
whose integration yields that
\begin{align}
   \fint_{B(x,r)} \left|\na^2 F_{\alpha} \right|^2 \leq C \left\{  1 + \fint_{B(x,r)} \sum_{k=1}^{n+1}|\nabla^2 f_k|^2 \right\}.  
  \label{eqn:OI03_2} 
\end{align}
Applying the Bochner formula, the eigenfunction and the Ricci curvature condition, we have
\begin{align}
 \frac{1}{2} \Delta (|\na f_k|^2)&=|\na^2 f_k|^2+\la \na f_k, \na \Delta f_k\ra+Rc(\na f_k, \na f_k) \notag\\
  &\geq |\na^2 f_k|^2 + \{ n-1-\lambda_k\} |\nabla f_k|^2.  \label{eqn:OI02_11}
\end{align}
Recall that $f_k^2+|\nabla f_k|^2 \leq 1+\tau$ by Lemma~\ref{lma:OH30_11}, and $1 \leq \lambda_k -n+1\leq 1+ \tau$ by Theorem~\ref{thm:OH30_5}. 
Thus we have
\begin{align*}
 |\na^2 f_k|^2  \leq \frac{1}{2}\Delta (|\na f_k|^2) + (1+ \tau).
\end{align*}
Take a cutoff function $\varphi$ supported on $B(x,4r)$ such that
\begin{align*}
  \varphi|_{B(x,2r)}=1, \quad |\na \varphi|<Cr^{-1}, \quad |\Delta \varphi|<Cr^{-2}. 
\end{align*}
Then we have
\begin{align*}
  &\quad \fint_{B(x,r)} |\na^2 f_k|^2 \\
  &\leq \fint_{B(x,r)} \varphi|\na^2 f_k|^2
   \leq Cr^{-n}\int_M \left\{\frac{1}{2}\Delta (|\na f_k|^2 \right\} \varphi+\fint_{B(x,4r)} \left\{1+ \tau \right\}\varphi\\
  &= Cr^{-n}\int_M\frac{1}{2} \left\{|\na f_k|^2-1 \right\} \Delta \varphi+ \fint_{B(x,4r)} \left\{ 1+ \tau\right\} \varphi\\
  &\leq C \left\{ 1 + r^{-n-2} \tau \right\}. 
 \end{align*}
 Plugging the above inequality into (\ref{eqn:OI03_2}) and multiplying both sides by $r^2$, we arrive at (\ref{eqn:OI02_5}).
 This finishes the proof of Step 3.\\
 
 In (\ref{eqn:OI02_4}) and (\ref{eqn:OI02_5}), the constant $C$ depends only on $n$.
 Let $c_n$ be a small constant such that $C \cdot c_n^2<\frac{1}{2}$. 
 Now fix $r \in (0, c_n \epsilon)$ and $\delta$ sufficiently small such that $C r^{-n}\tau(\delta | n)<0.5\epsilon$, we see that 
 (\ref{eqn:OI02_3}), (\ref{eqn:OI02_4}) and (\ref{eqn:OI02_5}) make sure the conditions (2), (3) and (4) in Definition~\ref{dfn:OH30_14} hold.
 Therefore, $\mathbf{F}|_{B(x,r)}$ is an $\epsilon$-splitting.   The proof of Theorem~\ref{thm:OH30_10} is complete. 
 \end{proof}

In the following theorem, we shall show that the $(n, \epsilon)$-splitting property is preserved in any smaller scale, up to modifying $\epsilon$ and using a lower triangular matrix transformation.
This theorem is inspired by the geometric transformation theorem in \cite{3}.
However, our argument is much simpler in the current special setting, and we also need to do necessary modifications as now we are considering eigenfunctions rather than harmonic functions.

\begin{thm}[\textbf{Transformation B}]
For each small $\eta$, there exists an $\epsilon_0=\epsilon_0(\eta, n)$ with the following properties. 

Suppose $x \in M$ and $(M^n, g)$ is a closed Riemannian manifold satisfying (\ref{eqn:OI01_4}). 
Suppose $\epsilon \in (0, \epsilon_0)$ and $r= c_n \epsilon$.  
Suppose $\mathbf{F}$ is chosen as in Theorem~\ref{thm:OH30_10} and $\delta$ is sufficiently small such that $\mathbf{F}|_{B(x,r)}$ is an $(n,\epsilon)$-splitting. 
Then for each $s \in (0, r)$,  there exists a lower triangular $(n \times n)$-matrix $T_s$ with positive diagonal entries such that 
\begin{itemize}
\item[(1).] $T_s  \mathbf{F} : B(x, r) \to \mathbb{R}^n$ is an $(n, \eta)$-splitting map on $B(x, s)$;
\item[(2).] $\fint_{B(x,s)} \langle\nabla(T_s \mathbf{F})_{\alpha}, \nabla(T_s \mathbf{F})_{\beta} \rangle=\delta_{\alpha \beta}$;
\item[(3).] $\left|T_s\circ T_{2s}^{-1}-I \right|\leq \eta$.
\end{itemize}

\label{thm:OI07_1}
\end{thm}

\begin{proof}

We call the transformation matrix $T_s$ with the properties (1)-(3) in Theorem~\ref{thm:OI07_1} as \textit{exact transformation}.
In order to prove the theorem, we need to show that  the exact transformation exists on each scale $s \in (0, r)$.

We argue by contradiction. 
If the theorem were wrong, we can find an $s \in (0, r)$ such that exact transformation does not exist on the scale $s$.
By the fact $\epsilon<<\eta$, it is clear that on the scale $r$, the exact transformation exists. 
Therefore, via a point-selecting process, we can find a critical scale $\bar{s}$ such that exact transformation exists 
for each scale $s \in [\bar{s}, r)$ and starts to fail to exist on the scale $\bar{s}$. \\

\textit{Step 1. For each $s \in (\bar{s}, r)$, we have 
\begin{align}
 \max\left\{|T_{s}\circ T_{2s}^{-1}-I|, \; |T_{2s}^{-1}\circ T_{s}-I| \right\}\leq C\sqrt{\eta}.   \label{eqn:OI07_2}
\end{align}
}

By condition (1) in the definition of $T_s$ and standard volume comparison,  we have
\begin{align}
 &\quad \left|\fint_{B(x, s)} \langle\na(T_{2s} \mathbf{F})_{\alpha}, \na(T_{2s} \mathbf{F})_{\beta}\rangle-\delta_{\alpha \beta} \right| \notag\\
 &\leq \fint_{B(x, s)} \left|\langle\na(T_{2s} \mathbf{F})_{\alpha}, \na(T_{2s} \mathbf{F})_{\beta}\rangle-\delta_{\alpha \beta} \right| \leq C \fint_{B(x,2s)} \left|\langle\na(T_{2s} \mathbf{F})_{\alpha}, \na(T_{2s} \mathbf{F})_{\beta}\rangle-\delta_{\alpha \beta} \right| \notag\\
 &\leq C \sqrt{\eta}.
\end{align}
Thus we can find a lower triangular matrix $U_{2s}$ such that
\begin{align}
& |U_{2s}-I|\leq C \sqrt{\eta}  \label{eqn:OI07_3}\\
& \fint_{B(x,s)}\left\langle\na(\widetilde{T}_{2s} \mathbf{F})_{\alpha},\,(\widetilde{T}_{2s} \mathbf{F})_{\beta} \right\rangle=\delta_{\alpha \beta},   \label{eqn:OI13_1}
\end{align}
where $\widetilde{T}_{2s}:=U_{2s} T_{2s}$. 
From condition (2) in the definition of $T_s$,  we know that
\begin{align}
\fint_{B(x,s)} \left\langle\na(T_{s} \mathbf{F})_{\alpha},\,(T_{s} \mathbf{F})_{\beta} \right\rangle=\delta_{\alpha \beta}.  \label{eqn:OI13_2}
\end{align}
Note that the $QR$-decomposition of the matrix $\fint_{B(x,{s})}\langle\na F_{\alpha}, \na F_{\beta} \rangle$ is unique, when we require the diagonal entries of the lower-triangular matrix are all positive. 
Thus it follows from (\ref{eqn:OI13_1}) and (\ref{eqn:OI13_2}) that $T_{s}=\widetilde{T}_{2s}$, which is equivalent to 
\begin{align}
   U_{2s}=T_{s}T_{2s}^{-1}.    \label{eqn:OI07_4}
\end{align}
Plugging (\ref{eqn:OI07_4}) into (\ref{eqn:OI07_3}), we obtain 
\begin{align}
|T_{s}\circ T_{2s}^{-1}-I|\leq C\sqrt{\eta}.   \label{eqn:OI07_6}
\end{align}
Composing $T_{2s}^{-1}$ and $T_{2s}$ on both sides, we obtain
\begin{align}
|T_{2s}^{-1}\circ T_{s}-I|\leq C\sqrt{\eta}.   \label{eqn:OI07_7}
\end{align}
It is clear that (\ref{eqn:OI07_2}) follows from the combination of (\ref{eqn:OI07_6}) and (\ref{eqn:OI07_7}). \\

\textit{Step 2. For each $s \in (\bar{s}, r)$, we have 
\begin{align}
  \max\left\{ \left|T_{s}^{-1}\circ T_{\bar{s}}\right|, \left| T_{\bar{s}} \circ T_{s}^{-1}\right| \right\} \leq  C \left\{ 1+ \left( \frac{s}{\bar{s}} \right)^{C \sqrt{\eta}} \right\}. \label{eqn:OI07_5}
\end{align}
}

Elementary linear algebra yields that 
\begin{align}
  |AB-I|\leq |A-I|+|B-I|+n|A-I|\cdot|B-I|    \label{eqn:OI07_8}
\end{align}
for every two $(n \times n)$-matrices $A$ and $B$.  
Note that
\begin{align*}
T_{s}\circ T_{2^l s}^{-1}=(T_{s}\circ T_{2^1 s}^{-1})\circ(T_{2^1 s}\circ T_{2^2s}^{-1})\cdots (T_{2^{l-1}s}\circ T_{2^l s}).
\end{align*}
Repeatedly applying (\ref{eqn:OI07_8}) on the above inequality, we obtain 
\begin{align*}
  \left|T_{s}^{-1}\circ T_{2^l s}-I \right|\leq \left\{1+(n+1)C\sqrt{\eta} \right\}^l-1, 
\end{align*}
which implies immediately that
\begin{align*}
  \left|T_{s}^{-1}\circ T_{\bar{s}} -I \right| \leq C \left\{ 1+ \left( \frac{s}{\bar{s}} \right)^{C \sqrt{\eta}} \right\}.
\end{align*}
Composing matrices in another direction, similar argument implies that
\begin{align*}
 \left| T_{\bar{s}} \circ T_{s}^{-1}-I\right|  \leq C \left\{ 1+ \left( \frac{s}{\bar{s}} \right)^{C \sqrt{\eta}} \right\}.
\end{align*}
It is clear that (\ref{eqn:OI07_5}) follows from the combination of the previous two inequalities. \\

\textit{Step 3.   We have
\begin{align}
      \frac{\bar{s}}{r} \leq \tau(\epsilon | n).  \label{eqn:OI08_3}
\end{align}
}

Note that $\mathbf{F}|_{B(x,r)}$ is an $\epsilon$-splitting, and $T_{\bar{s}} \mathbf{F}|_{B(x, \bar{s})}$ is an $(n,\eta)$-splitting for some $(n \times n)$-matrix $T$.
It follows from Definition~\ref{dfn:OH30_14} and volume comparison that
\begin{align}
 &\sup_{B(x, {s})} |\nabla \mathbf{F}|^2 \leq \sup_{B(x, \bar{s})} |\nabla \mathbf{F}|^2 \leq n + \epsilon;\label{eqn:OI07_1z}\\ 
 &\fint_{B(x,{s})}|\langle\na F_{\alpha}, \na F_{\beta} \rangle-\delta_{\alpha\beta}|< C \left(\frac{r}{{s}} \right)^n \fint_{B(x,r)}|\langle\na F_{\alpha}, \na F_{\beta} \rangle-\delta_{\alpha\beta}|<C \left(\frac{r}{{s}} \right)^n  \epsilon, \quad \forall \; 1 \leq \alpha \leq \beta \leq n;\label{eqn:OI07_2z}\\ 
 &{s}^2\fint_{B(x,{s})} |\na^2 \mathbf{F}|^2<C \left(\frac{r}{\bar{s}}\right)^{n-2} r^2\fint_{B(x,r)} |\na^2 \mathbf{F}|^2<C \left(\frac{r}{{s}}\right)^{n-2}  \epsilon^2 \label{eqn:OI07_3z} .
\end{align}
Note that $\epsilon<<\eta$.
Similar to the deduction of (\ref{eqn:OI07_3}), we can apply the estimate (\ref{eqn:OI07_2z}) to obtain a lower triangular matrix $U_{s}$ such that 
\begin{align}
|U_{s}-I|\leq C \left(\frac{r}{{s}} \right)^n  \epsilon, \label{eqn:OI07_4z}
\end{align} 
and the condition (2) in Theorem \ref{thm:OI07_1} is satisfied using $U_{s}$ as $ T_{s}$. To make sure that $U_{{s}}\mathbf{F}$ is an $(n,\eta)$-splitting and satisfies all the assumptions in  Theorem \ref{thm:OI07_1}, from (\ref{eqn:OI07_1z}), (\ref{eqn:OI07_2z}), (\ref{eqn:OI07_3z}) and (\ref{eqn:OI07_4z}), it suffices to have that $C \left(\frac{r}{{s}} \right)^{3n}  \epsilon^4<\eta$. That is, such transformation exists for scales $\frac{s}{r}\geq \tau(\epsilon | n)$. This guarantees that the critical scale satisfies $ \frac{\bar{s}}{r} \leq \tau(\epsilon | n)$, which is (\ref{eqn:OI08_3}).\\

\textit{Step 4.   Rescale the metric as
\begin{align}
  \bar{g} \coloneqq \bar{s}^{-2} g. 
 \label{eqn:OI08_6} 
\end{align}
Then for each fixed large constant $L$, we have
\begin{align}
     \frac{|B_{\bar{g}}(x, L)|_{\bar{g}}}{\omega_n L^n}   \geq 1 - \tau(\epsilon|n,L).  \label{eqn:OI08_4}
\end{align}
}

Since $Rc \geq n-1$, it follows from Bishop-Gromov volume comparison that 
\begin{align}
     \frac{|B(x, L\bar{s})|}{ n\omega_n \int_{0}^{L\bar{s}} \sin^{n-1} \theta d\theta} \geq  \frac{|M|}{(n+1) \omega_{n+1}} \geq 1-\frac{\delta}{(n+1) \omega_{n+1}}.   \label{eqn:OI08_5}
\end{align}
As $\bar{s} \in (0, c_n \epsilon)$ and $\delta \to 0$ as $\epsilon \to 0$,  it follows that
\begin{align*}
  \frac{|B(x, L\bar{s})|}{\omega_n (L\bar{s})^n} \geq \frac{n\omega_n \int_{0}^{L\bar{s}} \sin^{n-1} \theta d\theta}{\omega_n (L\bar{s})^n} \cdot \left\{ 1-\frac{\delta}{(n+1) \omega_{n+1}}\right\}
  \geq 1 - \tau(\epsilon|n,L), 
\end{align*}
which is equivalent to (\ref{eqn:OI08_4}). \\

\textit{Step 5.   Rescale the function as
\begin{align}
  \bar{\mathbf{F}} \coloneqq \bar{s}^{-1} T_{\bar{s}} \left\{ \mathbf{F}- \mathbf{F}(x) \right\}. 
 \label{eqn:OI07_9} 
\end{align}
 Then we have
 \begin{align}
    |\nabla_{\bar{g}} \bar{\mathbf{F}}|_{\bar{g}}(y) \leq C \left\{ 1+\rho^{C\sqrt{\eta}} \right\}
 \label{eqn:OI07_10}   
 \end{align}
 where $\rho=d_{\bar{g}}(y, x)$. 
}

Recall that $\mathbf{F}(x)=0 \in \R^{n}$.  It is clear that
\begin{align}
  \bar{\mathbf{F}} \coloneqq \bar{s}^{-1} T_{\bar{s}} \mathbf{F}=\bar{s}^{-1}  \left\{ T_{\bar{s}} \circ T_{\rho}^{-1} \right\} \circ T_{\rho} \mathbf{F}. 
  \label{eqn:OI08_1}
\end{align}
Note that $T_{\rho} \mathbf{F}$ is an $(n,\eta)$-splitting map on the scale $\rho$, it follows from item (2) in Definition~\ref{dfn:OH30_14} that
\begin{align}
   \left| \nabla (T_{\rho} \mathbf{F}) \right|= \left| \nabla_{g} (T_{\rho} \mathbf{F}) \right|_{g} \leq \sqrt{n+\eta} \leq \sqrt{2n} \leq C. 
  \label{eqn:OI08_2}
\end{align}
Plugging (\ref{eqn:OI07_5}) and (\ref{eqn:OI08_2}) into (\ref{eqn:OI08_1}), we obtain 
\begin{align*}
  \left| \nabla \bar{\mathbf{F}} \right| \leq C\bar{s}^{-1} \left\{ 1+ \left( \frac{s}{\bar{s}} \right)^{C \sqrt{\eta}} \right\}.
\end{align*}
In light of (\ref{eqn:OI07_9}), it is clear that (\ref{eqn:OI07_10}) is equivalent to the above inequality. \\

\textit{Step 6.    We have
\begin{align}
   \left|\Delta_{\bar{g}} \bar{\mathbf{F}} + \bar{\mathbf{F}} \right| \leq \tau(\epsilon|n). 
 \label{eqn:OI08_8} 
\end{align}
}

It follows from (\ref{eqn:OI08_7}) that
\begin{align*}
        \Delta \tilde{f}_{\alpha}=\frac{\Delta f_{\alpha}}{|f|} - \frac{\langle \nabla f_{\alpha}, \nabla |f|^2\rangle}{|f|^3}-\frac{f_k \Delta f_k}{|f|^3} f_{\alpha} + \frac{2|\nabla f|^2}{|f|^3} f_{\alpha}.  
\end{align*}
Note that
\begin{align*}
   \left| \nabla |f|^2 \right| \leq \sum_{\beta=1}^{n} |f_{\beta}||\nabla f_{\beta}| + |f_{n+1}||\nabla f_{n+1}| \leq \tau(\epsilon| n). 
\end{align*}
Combining the previous two steps, we obtain
\begin{align*}
   \left| \Delta \tilde{f}_{\alpha} + n \tilde{f}_{\alpha}\right| \leq \tau(\epsilon | n), \quad 1 \leq \alpha \leq n, 
\end{align*}
which implies that
\begin{align*}
  \left| \Delta \mathbf{F}+ n \mathbf{F} \right| \leq \tau(\epsilon | n). 
\end{align*}
It is clear that (\ref{eqn:OI08_8}) is equivalent to the above inequality by rescaling. \\

\textit{Step 7.    Derive the desired contradiction. 
}

For each small $\epsilon$, there exists $\bar{s} \in (0, c_n \epsilon)$ as a critical scale such that exact transformation starts to fail to exist on this scale.
In light of (\ref{eqn:OI08_4}), we can apply the Cheeger-Colding theory to obtain that
\begin{align}
        (M, x, \bar{g})  \longright{pointed-Gromov-Hausdorff}  (\hat{M}, \hat{x}, \hat{g}) \cong (\R^n, 0, g_{E}),     \label{eqn:OI09_1}
\end{align}
as $\epsilon \to 0$.  By the fact $\bar{\mathbf{F}}(x)=0$, the gradient estimate (\ref{eqn:OI07_10}),  and the Laplacian estimate (\ref{eqn:OI08_8}), we can apply Theorem~\ref{thm:OH31_4} to obtain that
\begin{align}
      &\bar{\mathbf{F}}   \stackrel{W_{loc}^{1,2}}{\longrightarrow} \hat{\mathbf{F}};  \label{eqn:OI09_2} \\
      &|\nabla_{\hat{g}} \hat{\mathbf{F}}|_{\hat{g}}(y) \leq C \left\{ 1+|y|^{C\sqrt{\eta}} \right\}, \quad \forall \; y \in \R^{n}; \label{eqn:OI09_3}\\
     &\Delta_{\hat{g}}  \hat{\mathbf{F}}=0, \; \textrm{in the distribution sense}.  \label{eqn:OI09_4}
\end{align}
The harmonicity of $\hat{\mathbf{F}}$ implies that $\hat{\mathbf{F}}$ is smooth,
thus we know $\displaystyle \Delta_{\hat{g}}  \hat{\mathbf{F}}=0$ in the traditional sense. 
Since the limit space is Euclidean, we can assume $y=(y^1, \cdots, y^n)$ being the coordinates.
Then it is clear that 
\begin{align*}
  \nabla_{\hat{g}} \hat{\mathbf{F}}= \left( \frac{\partial \hat{\mathbf{F}}}{\partial y^1}, \cdots,  \frac{\partial \hat{\mathbf{F}}}{\partial y^n} \right).
\end{align*}
Each component will also be harmonic and satisfies sublinear growth condition as $C\sqrt{\eta}<1$.   Then the proof of Liouville's theorem applies and we know that
\begin{align}
  \frac{\partial \hat{\mathbf{F}}}{\partial y^{\alpha}}=b_{\alpha}, \; \forall \; 1 \leq \alpha \leq n.    \label{eqn:OI09_11}
\end{align}
This means that $\hat{\mathbf{F}}$ is the linear function
\begin{align*}
    \hat{\mathbf{F}}(y)=\left(b_1 y^{1}, b_2 y^{2}, \cdots, b_n y^{n} \right). 
\end{align*}
By the choice of $\bar{s}$,  we know that $\bar{\mathbf{F}}$ is an $(n,\eta)$-splitting on scale $1$ and it satisfies
\begin{align}
 \fint_{B_{\bar{g}}(x, 1)} \left\{ \langle\na \bar{F}_{\alpha}, \na \bar{F}_{\beta} \rangle-\delta_{\alpha \beta} \right\} =0,    \label{eqn:OI09_7}
\end{align}
whose limit is
\begin{align}
  \fint_{B_{\hat{g}}(0, 1)} \left\{ \langle\na \hat{F}_{\alpha}, \na \hat{F}_{\beta} \rangle-\delta_{\alpha \beta} \right\}=b_{\alpha} b_{\beta} -\delta_{\alpha \beta}=0.  
  \label{eqn:OI09_12}
\end{align}
Thus $b_{\alpha}=\pm 1$ for each $\alpha \in \{1, 2, \cdots, n\}$.  
Consequently, it follows from (\ref{eqn:OI09_11}) and (\ref{eqn:OI09_12}) that 
\begin{align*}
  \fint_{B_{\hat{g}}(0, 1)} \left| \langle\na \hat{F}_{\alpha}, \na \hat{F}_{\beta} \rangle-\delta_{\alpha \beta} \right|=0.  
\end{align*}
In light of (\ref{eqn:OI09_2}),  the above equation implies that
\begin{align}
   \fint_{B(x, \bar{s})} \left| \langle\na F_{\alpha}, \na F_{\beta} \rangle-\delta_{\alpha \beta} \right| 
   = \fint_{B_{\bar{g}}(x, 1)} \left| \langle\na \bar{F}_{\alpha}, \na \bar{F}_{\beta} \rangle-\delta_{\alpha \beta} \right| 
   \leq \tau(\epsilon | n) < \eta. 
  \label{eqn:OI09_13}
\end{align}
It is also clear that
\begin{align}
&T_{\bar{s}} \mathbf{F}(0)=0; \label{eqn:OI09_8}\\
&\sup_{B(x, \bar{s})} \left| \nabla \left(T_{\bar{s}} \mathbf{F} \right) \right| \leq n+\tau(\epsilon | n)<n+\eta; \label{eqn:OI09_9}\\
&\bar{s}^2 \fint_{B(x, \bar{s})}  \left| \nabla^2 (T_{\bar{s}} \mathbf{F})\right| \leq \tau(\epsilon|n) <\eta.  \label{eqn:OI09_10}
\end{align}
Therefore, it follows from (\ref{eqn:OI09_13}), (\ref{eqn:OI09_8}), (\ref{eqn:OI09_9}) and (\ref{eqn:OI09_10}) that $T_{\bar{s}}$ is an exact transformation, with coefficients sharper than $\eta$.
By continuity, it is clear that $T_{\bar{s}}$ can be deformed to an exact transformation for some $s$ slightly smaller than $\bar{s}$.   This contradicts the assumption of $\bar{s}$. 
\end{proof}

\section{A diffeomorphism with bi-H\"older estimate}
\label{sec:canonical}

In this section we shall show that $\tilde{f}$ indeed provides a diffeomorphism with bi-H\"older estimate. 

\begin{prop}
For any $\ep>0$, there exists $\delta_0(n,\ep)>0$ with the following properties.

Suppose $(M^n, g)$ is a closed manifold satisfying (\ref{eqn:OI01_4}) for some $\delta \in (0, \delta_0)$. 
Suppose $x, y\in M$ and $d(x,y)<1$.
Then $\tilde{f}$ satisfies the bi-H\"older estimate
\begin{align}
   (1-\epsilon)d(x,y)^{1+\epsilon}\leq d\left(\tilde{f}(x), \tilde{f}(y) \right) \leq (1+\epsilon)d(x,y),      \label{eqn:OH31_7}
\end{align}
where $d\left(\tilde{f}(x), \tilde{f}(y) \right)$ is the distance induced by the canonical round metric on $S^{n}$. 
\label{prn:OI09_14}
\end{prop}

\begin{proof}

Fix $\xi=\tau(\delta | n)$ sufficiently small. 
By Theorem~\ref{thm:OH30_10} and Theorem~\ref{thm:OI07_1},  for any $x$, we can find a map $\mathbf{F}: M \to \R^n$
such that $\mathbf{F}|_{B(x,r)}$ is an $(n, \xi)$-splitting, where $r=c_n \xi$.  Furthermore, for each $s \in (0, r)$, 
there exists a lower triangular matrx $T_{s}$ such that $(T_{s} \mathbf{F})_{B(x, s)}$ is an $(n,\xi)$-splitting map. 

Note that $\mathbf{F}(x)=0$ from the definition of $(n, \xi)$-splitting.
From Lemma~\ref{lma:OI10_1} we know that
\begin{align}
  \left|T_{s}\mathbf{F}(y)-T_{s}\mathbf{F}(x) \right|\geq \left(1-\frac{\xi}{3} \right)d(x,y)
  \label{eqn:OI11_1}
\end{align}
for all $y \in  \partial B(x, s)$.  By (\ref{eqn:OI07_5}) in the proof of Theorem~\ref{thm:OI07_1},  we know 
\begin{align*}
  |T_{s}|\leq s^{-\xi}, 
\end{align*}
which implies that the smallest eigenvalue of $T_{s}^{-1}$ is at least $s^{\xi}$.  In other words, we have
\begin{align}
    T_{s}^{-1} (v) \geq s^{\xi} |v|, \quad \forall \; v  \in \R^n. 
 \label{eqn:OI11_2}    
\end{align}
Combining (\ref{eqn:OI11_1}) and (\ref{eqn:OI11_2}), we have
\begin{align*}
   |\mathbf{F}(y)-\mathbf{F}(x)|=\left| T_{s}^{-1} \left\{ T_{s} \mathbf{F}(y) - T_{s} \mathbf{F}(x) \right\}\right| \geq \left(1-\frac{\xi}{3} \right)d(x,y)^{1+\xi}.
\end{align*}

Since we also have the gradient estimate $|\nabla\mathbf{F}|\leq 1+\frac{\xi}{3}$, we now have for $y\in B(x, r)$ that
\begin{align*}
 \left(1-\frac{\xi}{3} \right)d(x,y)^{1+\xi}\leq |\mathbf{F}(x)-\mathbf{F}(y)|\leq \left(1+\frac{\xi}{3} \right)d(x,y).
\end{align*}
Recall that $q=(0,0,...,0,1)$ is the north pole. 
For the given $\xi>0$, there exists $r$ such that on $B(q, r)\subset S^n$ and we have
\begin{align*}
  \left(1-\frac{\xi}{3} \right)d(q,y)\leq |P(q)-P(y)|\leq \left(1+\frac{\xi}{3} \right) d(q,y), 
\end{align*}
where $P$ is a projection map defined as follows: 
\begin{align*}
   P: S^n &\mapsto  \R^{n},  \\
           (x_1,...,x_{n+1}) &\mapsto (x_1,...,x_n). 
\end{align*}
Since $\mathbf{F}$ is the composition of an orthogonal transformation and the projection map $P$, we have
\begin{align*}
 (1-\xi)d(x,y)^{1+\xi}\leq \left(1-\frac{\xi}{3} \right)|\mathbf{F}(x)-\mathbf{F}(y)|\leq d(\tilde{f}(x),\tilde{f}(y))\leq \left(1+\frac{\xi}{3} \right)|\mathbf{F}(x)-\mathbf{F}(y)|\leq (1+\xi)d(x,y).  
\end{align*}
This shows that (\ref{eqn:OH31_7}) is true for $d(x,y)<r=c\xi$. For $r\leq d(x,y)\leq 1$, from (\ref{eqn:OH30_7}), we know that by taking $\delta$ small enough we have that 
\begin{align*}
  (1-\xi)d(x,y)\leq d(\tilde{f}(x),\tilde{f}(y))\leq  (1+\xi)d(x,y), \quad \forall \;  d(x,y) \in (r, 1), 
\end{align*}
which implies (\ref{eqn:OH31_7}) immediately. 
\end{proof}

\begin{prop}
Same conditions as in Proposition~\ref{prn:OI09_14}.
Then $\tilde{f}$ is a diffeomorphism. 
\label{prn:OI09_15}
\end{prop}

\begin{proof}

\textit{Step 1. Injectivity.}

By the work of Petersen~\cite{13}, we know that $\tilde{f}$ is a $\tau(\delta | n)$-Gromov-Hausdorff approximation, which implies that
\begin{align*}
        \left| d(x,y) - d(\tilde{f}(x), \tilde{f}(y)) \right|< \tau(\delta | n)
\end{align*}
for every pair of points $x, y \in M$. 
Fix a point $w \in Im(\tilde{f})$.  It follows from the above inequality that 
\begin{align*}
       d(x, y) < \tau(\delta | n)<1
\end{align*}
for every pair of points $x, y \in \tilde{f}^{-1} (w)$. 
Therefore, the first inequality in (\ref{eqn:OH31_7}) applies and yields that
\begin{align*}
  (1-\epsilon) d(x,y)^{1+\epsilon} \leq d \left( \tilde{f}(x), \tilde{f}(y) \right)=d(w,w)=0, 
\end{align*}
which forces that $x=y$.  Therefore $\tilde{f}^{-1} (w)$ consists of a single point.
The injectivity of $\tilde{f}$ is proved. \\
 
 \textit{Step 2. Surjectivity.}
   
  The proof of surjectivity is deeply affected by the volume continuity proved by Colding \cite{6}. The argument below follows the route of the proof of volume continuity in Cheeger's note~\cite{1}.

  Recall that $v$ is a critical value of $\tilde{f}$ if $v=\tilde{f}(x)$ and $D \tilde{f}(x)$ degenerate for some $x \in M$. 
Let $E$ be the collection of all critical values of $\tilde{f}$.  From elementary differential topology,  we know that  
$\#\{\tilde{f}^{-1}(\cdot)\} \; mod \; 2$  is a constant on $S^n\backslash E$.

Firstly we show that there exists a  point $z_0\in S^n$ such that $\tilde{f}^{-1}(v)$ contains exactly one point in M.

In light of Theorem~\ref{thm:OH30_10},  for each $\ep>0$ and $\delta$ sufficiently small,  we have 
\begin{align*}
\mathbf{F}|_{B(x,r)}=((T\tilde{f})_1,...,(T\tilde{f})_n)|_{B(x,r)}, \quad r=c_n \epsilon
\end{align*}
is an $(n, \ep)$-splitting.   It follows from Definition~\ref{dfn:OH30_14} that 
\begin{align}
\fint_{B(x,r)} \left\{ \sum\limits_{1\leq \alpha \leq n}||\na (T\tilde{f})_i|^2-1|^2+\sum_{1\leq \alpha< \beta \leq n}|\langle\na (T\tilde{f})_{\alpha}, \na (T\tilde{f})_{\beta}\rangle|+r^2\sum\limits_{1\leq \alpha \leq n}|Hess \,\,(T\tilde{f})_{\alpha}|^2 \right\} \leq C\ep.
\label{eqn:OI10_5}
\end{align}
Let $W \subset B(x,r)$ be the union of closed geodesic balls $\overline{B(z, \rho)} \subset B(x,r)$ satisfying
\begin{align*}
   \fint_{B(z, \rho)} \left\{\sum\limits_{1\leq \alpha \leq n}||\na (T\tilde{f})_{\alpha}|^2-1|^2+\sum_{1\leq \alpha < \beta \leq n}|\langle\na (T\tilde{f})_{\alpha}, \na (T\tilde{f})_{\beta}\rangle|+\rho^2\sum\limits_{1\leq \alpha \leq n}|Hess \,\,(T\tilde{f})_{\alpha}|^2 \right\} > \xi. 
\end{align*}
By a standard covering argument(cf. Lemma 10.9 of~\cite{1}), we can find countably many disjoint balls $B(z_i, \rho_i)$ such that $W \subset \cup_{i=1}^{\infty} B(z_i, 5\rho_i)$. 
It follows that
\begin{align}
&\quad \xi Vol(B(z_i, \rho_i)) \notag\\
&< \int_{B(z_i, \rho_i)} \left\{\sum\limits_{1\leq \alpha \leq n}||\na (T\tilde{f})_{\alpha}|^2-1|^2+\sum_{1\leq \alpha < \beta \leq n}|\langle\na (T\tilde{f})_{\alpha}, \na (T\tilde{f})_{\beta}\rangle|+\rho^2\sum\limits_{1\leq \alpha \leq n}|Hess \,\,(T\tilde{f})_{\alpha}|^2 \right\}.
\label{eqn:OI10_3}
\end{align}
In light of volume comparison,  we have 
\begin{align}
  Vol(W) \leq   \sum_{i=1}^{\infty} Vol(B(z_i, 5\rho_i)) = \sum_{i=1}^{\infty} \frac{Vol(B(z_i, 5\rho_i))}{Vol(B(z_i, \rho_i))} \cdot Vol(B(z_i, \rho_i)) 
              \leq 5^{n} \sum_{i=1}^{\infty} Vol(B(z_i, \rho_i)) .  
\label{eqn:OI10_4}
\end{align}
Note that $\{B(z_i, \rho_i)\}_{i=1}^{\infty}$ is a disjoint collection and $\cup_{i=1}^{\infty} B(z_i, \rho_i) \subset B(x,r)$. It follows from the combination of  (\ref{eqn:OI10_3}) and (\ref{eqn:OI10_4}) that 
\begin{align*}
    &\quad Vol(W) \\
    &\leq \frac{5^n}{\xi} \int_{B(x,r)} \left\{\sum\limits_{1\leq \alpha \leq n}||\na (T\tilde{f})_{\alpha}|^2-1|^2+\sum_{1\leq \alpha < \beta \leq n}|\langle\na (T\tilde{f})_{\alpha}, \na (T\tilde{f})_{\beta}\rangle|+\rho^2\sum\limits_{1\leq \alpha \leq n}|Hess \,\,(T\tilde{f})_{\alpha}|^2 \right\} .
\end{align*}
Dividing both sides of the above inequality by $Vol(B(x,r))$ and applying (\ref{eqn:OI10_5}), we obtain 
\begin{align*}
 \frac{Vol(W)}{Vol(B(x,r))} \leq \frac{5^n}{\xi} \cdot C \epsilon < \tau(\epsilon| n), 
\end{align*}
where we appropriately choose $\xi$ in the last step.    Let $V=B(x,r) \backslash W$. Then we have 
\begin{align*}
&Vol(V)\geq \{1-\tau(\ep | n)\} Vol(B(x,r)), \\
&\fint_{B(z,l)} \left\{\sum\limits_{1\leq \alpha \leq n}||\na (T\tilde{f})_i|^2-1|^2+\sum_{1\leq \alpha< \beta \leq n}|\langle\na (T\tilde{f})_{\alpha}, \na (T\tilde{f})_{\beta}\rangle|+l^2\sum\limits_{1\leq \alpha \leq n}|Hess \,\,(T\tilde{f})_{\alpha}|^2 \right\} \leq \tau(\ep | n), 
\end{align*}
for all $z\in V$ and $B(z,l) \subset B(x,r)$. 
Thus we have
\begin{align}
l^2=\sum\limits_{1\leq \alpha  \leq n} \left\{(T\tilde{f})_{\alpha}-(T\tilde{f})_{\alpha}(z) \right\}+\tau(\ep | n)l^2\quad on \quad \partial B(z,l). 
\label{eqn:0I10_20z}
\end{align}
From (\ref{eqn:OH30_7}) we know that if $\tilde{f}(a)=\tilde{f}(b)$, we have  
\begin{align*}
|a-b|<\tau(\ep | n)
\end{align*}
when taking $\ep$ small enough such that $\tau(\ep |n )<\frac{r}{100}$.
In particular, we can choose $z_0\in B(x,\frac{r}{2})\cap V$ and $l_0=\frac{r}{2}$. Consider $\tilde{f}(z_0)\in S^n$. If $\tilde{f}(a)=\tilde{f}(z_0)$ for some $a \neq z_0$, then we have 
\begin{align*}
a \in B(z_0, l_0)\subset B(x,r). 
\end{align*}
Choose $T$ above with respect to $z_0$.  Namely,  rotate $\tilde{f}(z_0)$ to the north pole $q=(0, \cdots, 0, 1)$.
Combining this with (\ref{eqn:0I10_20z}) we get $1=\tau(\ep)$, which is a contradiction. Therefore we obtain $\tilde{f}^{-1}(\tilde{f}(z_0))=\{z_0\}$. 

Since we know that $T\tilde{f}(z_0)=(0,0,...,0,1)$, the first $n$ coordinate functions give a local chart near $T\tilde{f}(z_0)$. From the definition of V,  by sending  $l\to 0$, we know that $\na (T\tilde{f})_{\alpha}(z_0)$, 
$\alpha=\,1,2,...,n$ are almost orthogonal. This shows that $D(T\tilde{f})$ is non-degenerate at $z_0$. Combining with that $\tilde{f}^{-1}(\tilde{f}(z_0))=\{z_0\}$, we know that $\tilde{f}(z_0)\in S^n\setminus E$.
Thus there exists a point $\tilde{f}(z_0)=v\in S^n\backslash E$ such that $\tilde{f}^{-1}(v)$ contains exactly one point in M, which means that   
\begin{align*}
\#\{\tilde{f}^{-1}(\cdot)\} \equiv1 \; mod \; 2 
\end{align*}
on $S^n\backslash E$. In particular, we have $\#\{\tilde{f}^{-1}(\cdot)\} \geq 1$ on $S^n \backslash E$ and consequently 
\begin{align*}
    S^n \backslash E \subset Im(\tilde{f}). 
\end{align*}
Since $\tilde{f}$ is smooth and $M$ is compact, it is clear that $Im(\tilde{f})$ is a closed set. 
As $E$ is a measure-zero set, it follows that
\begin{align*}
   S^n=\overline{S^n \backslash E} \subset Im(\tilde{f}), 
\end{align*}
which means that $\tilde{f}$ is a surjective map.\\

\textit{Step 3. Nondegeneracy of $d\tilde{f}$.} 
  
  By now, we have showed that $\tilde{f}$ is a smooth homeomorphism map from $M$ to $S^n$.  In order to show $\tilde{f}$ being a diffeomorphism, it suffices to show the non-degeneracy of $d\tilde{f}: T_x M\rightarrow T_{\tilde{f}(x)} S^{n}$.
  By choosing natural orthonormal bases, we have
  \begin{align*}
     d\tilde{f}: \R^n \cong T_x M \to T_{\tilde{f}(x)} S^{n} \cong \R^n. 
  \end{align*}
  Thus we can regard $d\tilde{f}$ as an $(n \times n)$-matrix and $\det (d\tilde{f})(x)$ is well-defined.  
  The map $\tilde{f}$ is non-degenerate at $x$ if and only if $\det (d\tilde{f})(x) \neq 0$. 
  We shall prove this by contradiction. Suppose $\det(d\tilde{f})(x)=0$, then there exists a short geodesic $\gamma:[-\eta,\eta] \rightarrow M$ such that $\gamma(0)=x$ and $\left. \frac{d}{dt} \right|_{t=0} \tilde{f}(\gamma(t))=0$. 
Since $\tilde{f}$ is smooth, we denote that $a=\left. \frac{d^2}{dt^2} \right|_{t=0} \tilde{f}(\gamma(t))$, then from elementary calculus we know that $ d(\tilde{f}(\gamma(t)),\tilde{f}(\gamma(0)))=\frac{|a|}{2}t^2+o(t^2)$, which contradicts the bi-H\"older estimate (\ref{eqn:OH31_7}) when $t<<1$. 
\end{proof}

\section{Failure of bi-Lipschitz estimate}
\label{sec:example}

In this section we shall show the sharpness of the bi-H\"older estimate of the canonical diffeomorphism.
By constructing explicit examples, we show that the bi-Lipschitz estimate fails.  
The strategy is as follows, we first consider the football metrics on $S^n$ with two poles and find that the canonical eigenfunction-map has degenerate Jacobian at two poles. 
Then we delicately smooth the poles by spherical suspension metrics, and show that the Jacobian at some point  near the vertices of the smooth metric is very small. 
Of course, we need to keep the Ricci lower bound in the smoothing process.  
The observation dates back to \cite{26}\cite{18}, cf. section 1 of \cite{18}. 
The explicit construction of smoothened metric in this section also fills up some details of the construction in~\cite{20}.\\

Consider the rotationally symmetric metric $dr^2+\phi^2 d \boldsymbol{\theta}^2$, where $d \boldsymbol{\theta}^2$ is the standard round metric on $S^{n-1}$. 
Denote the corresponding coordinates in $S^{n-1}$ by $\theta_1, \theta_2,..., \theta_{n-1}$. 
At any given point $p$, we may  assume that $\left\{\frac{\partial}{\partial \theta_1}, \frac{\partial}{\partial \theta_2}, \cdots, \frac{\partial}{\partial \theta_{n-1}},\,\frac{\partial}{\partial r} \right\}$ is orthonormal at the point $p$ without loss of generality.
Then at point $p$ we have
\begin{align}
&Rc\left(\frac{\partial}{\partial \theta_i},\frac{\partial}{\partial \theta_j} \right)=\left((n-2)\frac{1-(\phi')^2}{\phi^2} -\frac{\phi''}{\phi} \right)\delta_{ij}, \label{eqn:OI20_2z}\\
&Rc\left(\frac{\partial}{\partial \theta_i},\frac{\partial}{\partial r} \right)=0, \label{eqn:OI20_21z}\\
&Rc\left(\frac{\partial}{\partial r},\frac{\partial}{\partial r} \right)=-(n-1)\frac{\phi''}{\phi}.\label{eqn:OI20_22z}
\end{align}
In order to guarantee $Rc\geq n-1$,  it suffices to show that
\begin{align}
 \phi''\leq-\phi, \quad \textrm{and} \quad \phi^2 + (\phi')^2 \leq 1.      \label{eqn:OI26_3}
\end{align}

Suppose we are allowed to study the problem in the category of singular manifolds,  we shall see that the bi-Lipschitz estimate fails. 
We consider the football metric, which can be written in polar coordinate as
\begin{align}
g_{c}=dr^2+c^2\sin^2r\; d \boldsymbol{\theta}^2,  \quad c \in (0, 1].   \label{eqn:OI17_1}
\end{align}
In this case,  $\phi=c \sin r$, it is clear that (\ref{eqn:OI26_3}) holds.  Thus $Rc \geq n-1$.
The volume of $(S^n, g_c)$ is close to the volume of standard sphere if $c$ is close to $1$. 
The Laplacian operator can be expressed as
\begin{align}
\Delta_{g_c} =\frac{\partial^2 }{\partial r^2}+\frac{1}{c^2 \sin^2 r}\Delta_{S^{n-1}}+(n-1) \cot r \frac{\partial}{\partial r}. 
\label{eqn:OI17_2}
\end{align}
Let
\begin{align}
    \alpha \coloneqq \frac{-n+\sqrt{n^2-4(n-1)(1-\frac{1}{c^2})}}{2}.    \label{eqn:OI21_2}
\end{align}
Let $h_1, h_2, \cdots, h_{n}$ be the first $n$ eigenfunctions of $(S^{n-1}, d \boldsymbol{\theta}^2)$. 
Separation of variables yields that the first $n+1$ eigenvalues and eigenfunctions of $-\Delta_{g_c}$ are 
\begin{align}
\begin{cases}
 \lambda_1=2, \quad &f_1=A\cos r;\\
 \lambda_{k+1}=\frac{-n+\sqrt{n^2-4(n-1)(1-\frac{1}{c^2})}}{2}+1+\frac{n-1}{c^2}, \quad &f_{k+1}=B \sin^{1+\alpha} r h_{k}(\boldsymbol{\theta}), \quad 1 \leq k \leq n.
\end{cases} 
\label{eqn:OI17_3}
\end{align}
Here $A$ and $B$ are constants adjusted for normalization conditions.  Then 
\begin{align*}
   \tilde{f}&=\frac{1}{\sqrt{A^2 \cos^2 r + B^2 \sin^{2+2\alpha} r}} \cdot \left( A\cos r, B\sin^{1+\alpha} r h_1(\boldsymbol{\theta}), \cdots, B \sin^{1+\alpha} r  h_{n}(\boldsymbol{\theta}) \right), \\
   d\tilde{f} \left( \frac{\partial}{\partial r} \right)&=\frac{\left( -A \sin r,  B(1+\alpha) \sin^{\alpha} r \cos r h_1(\boldsymbol{\theta}), \cdots, B(1+\alpha )\sin^{\alpha} r \cos r h_{n}(\boldsymbol{\theta}) \right) }{\sqrt{A^2 \cos^2 r + B^2 \sin^{2+2\alpha} r}}\\
    &\quad+\frac{\sin r \cos r \left( A^2-B^2(1+\alpha) \sin^{2\alpha} r\right)}{\left( \sqrt{A^2 \cos^2 r + B^2 \sin^{2+2\alpha} r} \right)^3} \cdot \left( A \cos r, B\sin^{1+\alpha} r h_1(\boldsymbol{\theta}), \cdots, B \sin^{1+\alpha} r h_{n}(\boldsymbol{\theta})  \right).
\end{align*}
It follows that
\begin{align*}
   \lim_{r \to 0^{+}} \left|  d\tilde{f} \left( \frac{\partial}{\partial r} \right)\right| =0, 
\end{align*}
which means $\tilde{f}$ does not induce a map with uniform bi-Lipschitz constant.   However, as the football metric is not smooth at two ends, it cannot serve as a genuine counter-example.
In order to achieve so, we need to smooth it at two poles while keeping the volume and curvature assumption.

\begin{prop}
 For each $n \geq 2$ and $\epsilon>0$, there exists a rotationally symmetric metric $g=dr^2+ \phi^2(r) d \boldsymbol{\theta}^2$ on $S^n$ satisfying (\ref{eqn:OI01_4}) and
 \begin{align}
     \inf_{x \in S^n}  \left|  d\tilde{f} \left( \frac{\partial}{\partial r} \right)\right|  \leq \epsilon,   \label{eqn:OI21_3}
 \end{align}
 no matter how small the number $\delta$ is in the inequality (\ref{eqn:OI01_4}). 
\label{prn:OI19_1}  
\end{prop}

\begin{proof}
We shall construct the metric by gluing a smooth small sphere metric to a football metric with two singular points.
The rotational symmetry will be preserved in the process.  Thus we are actually gluing two functions
\begin{align*}
    \phi_{g_c}=c \sin r, \; 0<c <1;   \quad \textrm{and} \quad \phi_{a^{-2} g_{round}}=\frac{\sin\left(ar\right)}{a}, \; a > 1. 
\end{align*}
It is clear that $Rc\geq n-1$ in either of the above cases.

Choose a small positive constant $\xi$ and define
\begin{align}
    a \coloneqq \frac{\sqrt{1-c^2 \cos^2 \xi}}{c \sin \xi}, \quad \rho \coloneqq \frac{\arccos (c \cos \xi)}{a}. 
\end{align}
Then we have
\begin{align}
 c \sin\xi =\frac{\sin\left(a \rho \right)}{a}, \quad c \cos \xi=\cos\left(a \rho \right).\label{eqn:OI20_4z}
\end{align}
It is clear that $a \to +\infty$ and $\rho \to 0^{+}$, as $\xi \to 0^{+}$.

Take a smooth function 
\begin{align*}
\psi:\R\to \R,\,\,\psi(x)=\left\{\begin{array}{ccc}0, \,\,\,\,\quad\forall x\leq 0;\\1, \,\,\,\,\quad\forall x\geq 1.\end{array}\right.
\end{align*}
Denote
 \begin{align}
 l(t) \coloneqq -\left\{\left(1-\psi\left(\frac{t-\rho}{\zeta}\right)\right)a\,\sin\left(a\,t\right)
+\psi\left(\frac{t-\rho}{\zeta}\right)c\,\sin\left(t-\rho-\zeta+\xi\right)\right\}. 
\label{eqn:OI20_5z}
\end{align} 
From its definition, for each integer $k \geq 0$, we have
\begin{align}
&l^{(2k)}(\rho)=(-1)^{k+1} a^{2k+1} \sin (a \rho),  \quad l^{(2k)}(\rho+\zeta)=(-1)^{k+1} c \sin \xi; \label{eqn:OI26_1}\\
&l^{(2k+1)}(\rho)=(-1)^{k+1} a^{2k+2} \cos (a \rho),  \quad  l^{(2k+1)}(\rho+\zeta)= (-1)^{k+1} c \cos \xi.   \label{eqn:OI26_2}
\end{align}
Then we construct the function $\phi=\phi_{\xi,\zeta,\kappa}$ as follows.  
\begin{itemize}
\item[(1).]  If $r \in [0,\rho]$, then 
       \begin{align}
        \phi \coloneqq \frac{\sin(ar)}{a}. 
        \label{eqn:OI27_3}   
      \end{align}
\item[(2).]  If $r \in [\rho,\rho+\zeta]$, then 
     \begin{align}
        \phi \coloneqq \frac{\sin(a\rho)}{a}+(r-\rho)\cos(a\rho)+\int_{s=\rho}^{r}\int_{t=\rho}^{s}l(t). 
     \label{eqn:OI27_1}   
     \end{align}
\item[(3).]  If $r \in [\rho+\zeta,\rho+\zeta+\kappa]$, then 
        \begin{align}
          \phi&\coloneqq c \;\sin \xi+c(r-\rho-\zeta)\cos \xi+\left(1-\psi\left(\frac{r-(\rho+\zeta)}{\kappa}\right)\right)\int_{s=\rho}^{r}\int_{t=\rho}^{s}l(t) \notag\\
                &\quad+\psi\left(\frac{r-(\rho+\zeta)}{\kappa}\right)\int_{s=\rho+\zeta}^{r}\int_{t=\rho+\zeta}^{s}l(t)+\left(1-\psi\left(\frac{r-(\rho+\zeta)}{\kappa}\right)\right)c\zeta\,\cos \xi. 
                \label{eqn:OI27_2}
        \end{align}
\item[(4).]  If $r \in [\rho+\zeta+\kappa,  \frac{\pi}{2} + \rho + \zeta -\xi]$,  then 
       \begin{align}
        \phi \coloneqq c\,\sin(r-\rho-\zeta+\xi). 
        \label{eqn:OI27_4}   
      \end{align}
\end{itemize}
In light of (\ref{eqn:OI20_4z}), (\ref{eqn:OI26_1}) and (\ref{eqn:OI26_2}), direct calculation implies that $\phi$ is smooth at $r=\rho, \rho+\zeta$ and $\rho+\zeta+\kappa$. 
Consequently, $\phi$ is smooth on $[0, \frac{\pi}{2} + \rho + \zeta -\xi)$.  By reflection, we can smooth the other pole similarly.  Therefore, the metric $dr^2 + \phi^2  d \boldsymbol{\theta}^2$ is a smooth metric on $S^n$. 

\textbf{Claim:\;}\textit{For each fixed $\xi \in (0, 1)$ and sufficiently small  $\kappa$ and $\zeta << \kappa$,   the smooth function $\phi=\phi_{\xi, \zeta, \kappa}$ constructed above satisfies
\begin{align}
  &\phi''\leq-(1-\xi)\phi,    \label{eqn:OI26_5}\\
  &(\phi')^2+(1-\xi)\phi^2\leq 1.     \label{eqn:OI26_6}
\end{align}
}

By discussion at the beginning of this section, it is clear that (\ref{eqn:OI26_5}) and (\ref{eqn:OI26_6}) are satisfied in cases (1) and (4).  
Thus we shall only focus on the proof of (\ref{eqn:OI26_5}) and (\ref{eqn:OI26_6}) in cases (2) and (3). 

We first show (\ref{eqn:OI26_5}). 
On the interval $ [\rho,\rho+\zeta]$, it follows from (\ref{eqn:OI20_5z})  and (\ref{eqn:OI27_1}) that
\begin{align*}
  -\phi''&=-l(r)=(1-\psi) a \sin (a r) + \psi c \sin (r-\rho -\zeta + \xi)= (1-\psi) a \sin (a \rho) + \psi c \sin \xi + O(\zeta), \\
  \phi&=\frac{\sin (a \rho)}{a} + O(\zeta). 
\end{align*}
Applying (\ref{eqn:OI20_4z}) on the above equations, we have
\begin{align*}
  -\phi''=\frac{\sin (a \rho)}{a} + (1-\psi) \left(a-\frac{1}{a} \right) \sin (a \rho) + O(\zeta)
   \geq \frac{\sin (a \rho)}{a} + O(\zeta) \geq (1-\xi) \phi.
\end{align*}
Thus we obtain (\ref{eqn:OI26_5}) in case (2).
On the interval $[\rho+\zeta,\rho+\zeta+\kappa]$,  the second derivative $\phi''$ can be expressed as
\begin{align*} 
&\quad l(r)+\left\{\psi\left(\frac{r-(\rho+\zeta)}{\kappa}\right)\left(\int_{s=\rho+\zeta}^{r}\int_{t=\rho+\zeta}^{s}l(t)-\int_{s=\rho}^{r}\int_{t=\rho}^{s}l(t)\right)+\left(1-\psi\left(\frac{r-(\rho+\zeta)}{\kappa}\right)\right)c\zeta\,\cos \xi \right\}''\\
&=l(r)+\frac{1}{\kappa^2}\psi''\left(\frac{r-(\rho+\zeta)}{\kappa}\right)\left(\int_{s=\rho+\zeta}^{r}\int_{t=\rho+\zeta}^{s}l(t)-\int_{s=\rho}^{r}\int_{t=\rho}^{s}l(t)-c\zeta\,\cos \xi \right)\notag\\
&\quad \quad +2\frac{\psi'\left(\frac{r-(\rho+\zeta)}{\kappa}\right)}{\kappa}\left(\int_{t=\rho+\zeta}^{r}l(t)-\int_{t=\rho}^{r}l(t)\right)\\
&=l(r) + O\left(\frac{\zeta^2}{\kappa^2} \right) + O\left(\frac{\zeta}{\kappa} \right). 
\end{align*}
Since $\zeta << \kappa$ according to our choice, it follows from (\ref{eqn:OI20_5z}),  (\ref{eqn:OI27_2}) and the above inequality that 
\begin{align*}
&-\phi''=c \sin \xi + O(\kappa)+O\left(\frac{\zeta}{\kappa} \right), \\
&\phi=c \sin \xi + O(\kappa). 
\end{align*}
Recalling that both $\kappa$ and $\frac{\zeta}{\kappa}$ are very small,  we immediately obtain (\ref{eqn:OI26_5}) holds on $[\rho+\zeta, \rho+\zeta+\kappa]$, i.e., in case (3).

We move on to show (\ref{eqn:OI26_6}).  Taking derivative of (\ref{eqn:OI27_1}) yields that
\begin{align*}
\phi'=\cos (a \rho) + \int_{t=\rho}^{r} l(t), \quad \textrm{on} \quad [\rho, \rho+\zeta]. 
\end{align*}
Since $\zeta$ is sufficiently small, it follows from (\ref{eqn:OI20_4z}) and (\ref{eqn:OI27_1}) that
\begin{align*}
   (\phi')^2 + (1-\xi) \phi^2 \leq  (\phi')^2 + \phi^2 =\left\{ \frac{\sin (a\rho)}{a}\right\}^2 + \cos^2 (a \rho) + O(\zeta)=c^2+O(\zeta)<1, 
\end{align*}
which shows (\ref{eqn:OI26_6}) in case (2).  
On the interval $[\rho+\zeta,\rho+\zeta+\kappa]$, by (\ref{eqn:OI27_2}) we have
\begin{align*}
 \phi'
 &=c \cos \xi + \int_{t=\rho}^{r} l(t) \\
 &\quad+ \left\{\psi\left(\frac{r-(\rho+\zeta)}{\kappa}\right)\left(\int_{s=\rho+\zeta}^{r}\int_{t=\rho+\zeta}^{s}l(t)-\int_{s=\rho}^{r}\int_{t=\rho}^{s}l(t)\right)+\left(1-\psi\left(\frac{r-(\rho+\zeta)}{\kappa}\right)\right)c\zeta\,\cos \xi \right\}'\\
 &=c \cos \xi + O(\zeta + \kappa); \\
 \phi
 &=c \sin \xi + O(\zeta+\kappa). 
\end{align*}
It follows that 
\begin{align*}
 (\phi')^2 + (1-\xi) \phi^2=  (c \cos \xi)^2 + (1-\xi) (c \sin \xi)^2 + O(\zeta+\kappa) \leq c^2 + O(\zeta+\kappa) <1, 
\end{align*}
which proves (\ref{eqn:OI26_6}) in case (3).    The proof of the Claim is complete. \\

By (\ref{eqn:OI26_5}) and (\ref{eqn:OI26_6}), it follows from (\ref{eqn:OI20_2z})-(\ref{eqn:OI20_22z}) that  $Rc \geq (n-1)(1-\xi)$ under the metric $dr^2 + \phi^2 d \boldsymbol{\theta}^2$.  
Let $g_{\xi,c}=(1-\xi)^{-1} \left\{ dr^2 + \phi^2 d \boldsymbol{\theta}^2 \right\}$.  Then we have 
\begin{align}
  Rc \geq n-1, \quad \textrm{under} \; g_{\xi,c}.     \label{eqn:OI27_5}
\end{align}
Recall that $g_c=dr^2+c^2\sin^2 r d \boldsymbol{\theta}^2$. Then it is clear that
\begin{align}
&\lim_{\xi \to 0^+}d_{GH}\left(\left(S^n,g_{\xi,c}\right),\left(S^n, g_c \right)\right)=0,   \label{eqn:OI27_9}\\
&\lim_{c\to 1^-}d_{GH}\left(\left(S^n,g_{round}\right),\left(S^n, g_c \right) \right)=0. \label{eqn:OI27_10}
\end{align}
The spectrum of the standard sphere is
\begin{align*}
 \lambda_{1}=\lambda_{2}=\cdots=\lambda_{n+1}=n<\lambda_{n+2}=2(n+1) \leq \cdots. 
\end{align*}
By continuity of spectrum under the Gromov-Hausdorff convergence(cf.~\cite{21}),  we know that  the spectrum of $\left(S^n, g_c \right)$ is close to that on $\left(S^n,g_{round} \right)$.
Therefore,  when $c$ is close to $1$,  there are exactly $n+1$ eigenvalues strictly less than $n+2$. More precisely, we have
\begin{align*}
&\lambda_{1,S_c^n}=n; \\
&\lambda_{k+1,S_c^n}=\alpha+1+\frac{n-1}{c^2}, \quad \forall \; k \in \{1, 2, \cdots, n\};\\
&\lambda_{n+2, S_c^n} \geq n+2. 
\end{align*}
Here $\alpha$ is defined in (\ref{eqn:OI21_2}). 
Take eigenfunctions $f_{1,\xi,c},\,\,f_{2,\xi,c}, \cdots, f_{n+1,\xi,c}$ on $\left(S^n,g_{\xi,c}\right)$ corresponding to the first $(n+1)$-eigenvalues with 
\begin{align*}
&\fint_{(S^n,g_{\xi,c})}\left|f_{k,\xi,c}\right|^2dvol_{\left(S^n,g_{\xi,c}\right)}=\frac{1}{n+1}, \quad \forall \; 1 \leq k \leq n+1;\\
&\fint_{\left(S^n,g_{\xi,c}\right)}f_{j,\xi,c}f_{k,\xi,c} dvol_{\left(S^n,g_{\xi,c}\right)}=0, \quad  \forall \; 1 \leq j<k \leq n+1. 
\end{align*}
From Lemma~\ref{lma:OH30_11} and Theorem \ref{thm:OH31_4},  we know that  these eigenfunctions sub-converge in $W^{1,2}$-sense to the eigenfunctions
\begin{align*}
  C_A \cos r, \quad C_B \sin^{1+\alpha} r \left\{ \sum_{k=1}^{n} p_{kj} h_j(\boldsymbol{\theta})\right\},  \quad k=1, 2, \cdots, n
\end{align*}
on $\left(S^n, g_{c} \right)$, as $\xi \to 0^+$. 
Here $p_{kj}$ is an $(n \times n)$-orthonormal matrix, $\{h_k\}$ are eigenfunctions corresponding to eigenvalue $n-1$ on $(S^{n-1}, d \boldsymbol{\theta}^2)$. 
Note that the following equations hold:
 \begin{align}
 \lim_{r \to 0^+} \sin^{1+\alpha} r h_k(\boldsymbol{\theta})(p)=0, \quad
 \lim_{r \to 0^+} \left| \na\left(\sin^{1+\alpha} r  h_k(\boldsymbol{\theta})\right) \right| =0, \quad k=1, 2, \cdots, n. 
 \label{eqn:OI27_8} 
 \end{align}
For each fixed small number $\eta$, $c=1-\eta$ and $\xi$ sufficiently small,  by (\ref{eqn:OI27_8}) and the $W^{1,2}$-convergence along (\ref{eqn:OI27_9}), we can find a point $(\rho, \boldsymbol{\theta})$ with $\left|\rho\right|<\eta$ such that
\begin{align*}
&\left|\na f_{1,\xi,c}\left(\rho,\boldsymbol{\theta} \right)\right|^2\leq  \eta^2;\\
&\left|f_{k,\xi,c}\left(\rho, \boldsymbol{\theta} \right)\right|^2+\left|\na f_{k,\xi,c}\left(\rho, \boldsymbol{\theta} \right)\right|^2\leq \eta^2, \quad k=2, 3, \cdots, n+1. 
\end{align*}
Thus we have
\begin{align}
 \left|d \left(\frac{f_{1,\xi,c}}{\sqrt{\sum_{j=1}^{n+1}f_{j,\xi,c}^2}},\frac{f_{2,\xi,c}}{\sqrt{\sum_{j=1}^{n+1}f_{j,\xi,c}^2}}\,,...\,,\frac{f_{n+1,\xi,c}}{\sqrt{\sum_{j=1}^{n+1}f_{j,\xi,c}^2}} \right) \left(\frac{\partial}{\partial r}  \right) \right|\left(\rho, \boldsymbol{\theta} \right)\leq  C \eta.
\label{eqn:OI27_6} 
\end{align}
However,  in light of (\ref{eqn:OI27_9}) and (\ref{eqn:OI27_10}), the metric $g_{\xi, c}=g_{\xi_{\eta}, 1-\eta}$ satisfies
\begin{align*}
 \lim_{\eta \to 0^+}d_{GH}\left(\left(S^n,g_{round}\right),\left(S^n, g_{\xi, c} \right) \right)=0. 
\end{align*} 
By (\ref{eqn:OI27_5}), the above line is equivalent (cf. Theorem~\ref{thm:OH30_5}) to
\begin{align}
     \lim_{\eta \to 0^+} Vol \left(S^n, g_{\xi, c} \right)=Vol \left(S^n,g_{round}\right).    \label{eqn:OI27_7}
\end{align}
Therefore, (\ref{eqn:OI21_3}) follows immediately from the combination of (\ref{eqn:OI27_6}) and (\ref{eqn:OI27_7}). 
\end{proof}

Now we finish the proof of the main theorem. 

\begin{proof}[Proof of Theorem~\ref{thm:main}]
In Proposition~\ref{prn:OI09_15}, we show that the canonical map $\tilde{f}$ defined in (\ref{eqn:OH30_3}) is a diffeomorphism.
The uniform bi-H\"older estimate (\ref{eqn:OH30_4}) is proved in Proposition~\ref{prn:OI09_14}. 
Finally, the sharpness of the bi-H\"older estimate is proved in Proposition~\ref{prn:OI19_1} by example construction. 
Therefore, the proof of Theorem~\ref{thm:main} is complete. 
\end{proof}

\begin{rem}
In Perelman's paper~\cite{20}, as a cone has $Rc(T,T)=0$, we need to change the region $t\in[0,\frac{T}{2}]$ to a smooth space with positive Ricci curvature. Originally near the singular point the space is a nonsymmetrical cone with metric $g=dt^2+\frac{t^2}{100}(dx^2+dy^2+(1-\gamma(t))^2dz^2)$, where $\gamma',\gamma''>0$ for $t\in (\frac{T}{2},T)$ and $\gamma(t)=0,\,\,t\in[0,\frac{T}{2}]$. As when we denote $g=dt^2+A^2(dx^2+dy^2)+C^2dz^2$, using Perelman's notation in ~\cite{20} we would have 
\begin{align}
&\frac{Rc(X,X)}{\|X\|^2}=\frac{Rc(Y,Y)}{\|Y\|^2}=-\frac{A''}{A}-\frac{(A')^2}{A^2}-\frac{A'C'}{AC}+\frac{1}{A^4}(-2C^2+4A^2)\label{Ric X},\\
&\frac{Rc(Z,Z)}{\|Z\|^2}=-\frac{C''}{C}-\frac{2A'C'}{AC}+\frac{2C^2}{A^4}\label{Ric Z},\\
&\frac{Rc(T,T)}{\|T\|^2}=-\frac{A''}{A}-\frac{B''}{B}-\frac{C''}{C}.\label{Ric T}
\end{align} 
The subtlety here is that now the space is nonsymmetric and we need to guarantee that the whole space has positive Ricci curvature. As in Proposition \ref{prn:OI19_1}, we can set  $g_{xx}''(t)=g_{yy}''(t)=\psi_1(t) (\frac{\sin(at)}{a})''$ and 
$g_{zz}''(t)=\psi_2(t)(\frac{\sin(at)}{a})''+\frac{1}{10}(1-\psi_2(t))(t-t\gamma(t))''$, which are all negative, to smoothly change $\tilde{g}=dt^2+(\frac{\sin(at)}{a})^2g_{S^3}$ to $g=dt^2+(\frac{t}{10}+\ep_1)^2(dx^2+dy^2)+(\frac{t}{10}(1-\gamma(t))+\ep_2)^2dz^2$ and the paste point is $t=\frac{T}{2}+\ep_3$, where $\ep_1,\ep_2,\ep_3$ can be chosen sufficiently small. Since we know that in the process $A,B,C,A',B',C'$ remain almost the same and as in the sphere $\tilde{g}=dt^2+(\frac{\sin(at)}{a})^2g_{S^3}$\begin{align}-\frac{(A')^2}{A^2}-\frac{A'C'}{AC}+\frac{1}{A^4}(-2C^2+4A^2)=-\frac{2A'C'}{AC}+\frac{2C^2}{A^4}=2a^2>0\end{align} while $-\frac{A''}{A}=-\frac{B''}{B}\geq 0$, $-\frac{C''}{C}>0$ in the whole process, where we use that $\gamma''(\frac{T}{2}+\ep_3),\gamma'(\frac{T}{2}+\ep_3)>0$, the positivity of Ricci curvature in the smoothing process is guaranteed. This could generate remaining terms $\ep_1,\ep_2$ in the metric, where $\ep_1,\ep_2$ could be sufficiently small and this would not affect the positivity of Ricci curvature in the region 
$t\in [\frac{T}{2}+\ep_3,+\infty)$.  
\label{rmk:OI28_1}
\end{rem}

\begin{rem}
By slight rescaling, (\ref{eqn:OI01_4}) is equivalent to the following condition
\begin{subequations}
  \begin{empheq}[left = \empheqlbrace \,]{align}
    &Rc \geq n-1-\delta, \label{eqn:OI28_1a} \\
    &Vol(M, g)=(n+1) \omega_{n+1}.    \label{eqn:OI28_1b}
  \end{empheq}
\label{eqn:OI28_1}  
\end{subequations}
\hspace{-3mm}
The same conclusion in Theorem~\ref{thm:main} holds if we replace (\ref{eqn:OI01_4}) by (\ref{eqn:OI28_1}).
In light of the structure theory of manifolds with bounded Ricci curvature integral,  the condition (\ref{eqn:OI28_1a}) can be replaced by $\int \{Rc-(n-1)\}_{-}^{p}<\delta(n,p)$ for some $p>\frac{n}{2}$. 
Such generalization in the Ricci flow case was already achieved in~\cite{30}. 
\label{rmk:OI28_2}
\end{rem}


\begin{thebibliography}{2}

\bibitem[Ch]{1}J. Cheeger,
\emph{Degeneration of Riemannian metrics under Ricci curvature bounds},
Lezioni Fermiane. [Fermi Lectures] Scuola Normale Superiore, Pisa, 2001. ii+77 pp.

\bibitem[CC1]{26}J. Cheeger and T.H. Colding, 
\emph{Lower bounds on Ricci curvature and the almost rigidity of warped products},
Ann. of Math. (2) 144 (1996), no. 1, 189-237.

\bibitem[CC2]{2}J. Cheeger and T.H. Colding, 
\emph{On the structure of spaces with Ricci curvature bounded below. I},
J. Differential Geom. 46 (1997), no. 3, 406-480.


\bibitem[CC3]{25}J. Cheeger and T.H. Colding, 
\emph{On the structure of spaces with Ricci curvature bounded below. II},
J. Differential Geom. 54 (2000), no. 1, 13-35.


\bibitem[CC4]{21}J. Cheeger and T.H. Colding, 
\emph{On the structure of spaces with Ricci curvature bounded below. III},
J. Differential Geom. 54 (2000), no. 1, 37-74.

\bibitem[CJN]{3}J. Cheeger, W. Jiang, and A. Naber,
\emph{Rectifiability of singular sets of noncollapsed limit spaces with Ricci curvature bounded below}, 
Ann. of Math. 193 (2021), no. 2, 407-538.


\bibitem[CW]{4}
X.X. Chen and B. Wang,
\emph{Space of Ricci flows (II)---Part A: Moduli of singular Calabi-Yau spaces},
Forum Math. Sigma 5 (2017), Paper No. e32, 103 pp with ``Further details" available at: http://staff.ustc.edu.cn/~topspin/paper/htfurtherdetails.pdf. 



\bibitem[C1]{5}
T.H. Colding, 
\emph{Shape of manifolds with positive Ricci curvature},
Invent. Math. 124 (1996), no. 1-3, 175-191.


\bibitem[C2]{6}
T.H. Colding, 
\emph{Large manifolds with positive Ricci curvature},
Invent. Math. 124 (1996), no. 1-3, 193-214.


\bibitem[CM]{18}
 T.H. Colding and W.P. Minicozzi II,
\emph{Harmonic functions with polynomial growth},
J. Differential Geom. 46 (1997), no. 1, 1-77.


\bibitem[CN]{32}
T.H. Colding and A. Naber,
\emph{Sharp H\"older continuity of tangent cones for spaces with a lower Ricci curvature bound and applications}. 
Ann. of Math. 176 (2012), no. 2, 1173-1229.



\bibitem[D]{8}Y. Ding,
\emph{Heat kernels and Green's functions on limit spaces},
Comm. Anal. Geom. 10 (2002), no. 3, 475-514.





\bibitem[HW]{17}
S.S. Huang and B. Wang,
\emph{Rigidity of vector valued harmonic maps of linear growth}, 
Geom. Dedicata. 202 (2019), 357-371.




\bibitem[L]{11}
A. Lichnerowicz,
\emph{G\'{e}om\'{e}trie des groupes de transformations},  Travaux et Recherches Mathématiques, III. Dunod, Paris 1958 ix+193 pp.





\bibitem[M]{24}
S.B. Myers, 
\emph{Riemannian manifolds with positive mean curvature},
Duke Math. J. 8 (1941), 401-404.



\bibitem[MW]{30}
Y.Q. Ma, B. Wang,  
\emph{Ricci integrals, local functionals, and the Ricci flow}, preprint, arXiv: 2109.02449v1. 


\bibitem[O]{12}
M. Obata,
\emph{Certain conditions for a Riemannian manifold to be isometric with a sphere},
J. Math. Soc. Japan. 14 (1962), 333-340.





\bibitem[P]{20}
G. Perelman,
\emph{A complete Riemannian manifold of positive Ricci curvature with Euclidean volume growth and nonunique asymptotic cone}, Comparison geometry (Berkeley, CA, 1993-94), 165-166,
Math. Sci. Res. Inst. Publ., 30, Cambridge Univ. Press, Cambridge, 1997.




\bibitem[Pet]{13}
P. Petersen,
\emph{On eigenvalue pinching in positive Ricci curvature},
Invent. Math. 138 (1999), no. 1, 1-21.





\bibitem[W1]{22}
B. Wang,
\emph{The local entropy along Ricci flow---Part A: the no-local-collapsing theorems}, 
Camb. J. Math. 6 (2018), no. 3, 267-346.





\bibitem[W2]{23}
B. Wang, \emph{The local entropy along Ricci flow---Part B: the pseudo-locality theorems}, preprint, arXiv: 2010.09981v1.











\end{thebibliography}

\vskip10pt

Bing Wang,  Institute of Geometry and Physics, and Key Laboratory of Wu Wen-Tsun Mathematics,
School of Mathematical Sciences, University of Science and Technology of China, No. 96 	Jinzhai Road, Hefei, Anhui Province, 230026, China. 

Email: topspin@ustc.edu.cn.

\vskip10pt
Xinrui Zhao, Department of Mathematics, Massachusetts Institute of Technology, 77 Massachusetts Avenue, Cambridge, MA 02139-4307, USA. 

Email: xrzhao@mit.edu.

\end{document}